\documentclass{article}
\usepackage{amsfonts}
\usepackage{amsmath}

\newtheorem{theorem}{Theorem}

\newtheorem{note}[theorem]{Note}

\newtheorem{convention}[theorem]{Convention}

\newtheorem{corollary}[theorem]{Corollary}

\newtheorem{definition}[theorem]{Definition}

\newtheorem{lemma}[theorem]{Lemma}

\newtheorem{proposition}[theorem]{Proposition}
\newtheorem{remark}[theorem]{Remark}

\newenvironment{proof}[1][Proof]{\textbf{#1.} }{\ \rule{0.5em}{0.5em}}
\numberwithin{equation}{section}

\input{tcilatex}

\begin{document}

\title{Univariate error function based neural network approximation}
\author{George A. Anastassiou \\
Department of Mathematical Sciences\\
University of Memphis\\
Memphis, TN 38152, U.S.A.\\
ganastss@memphis.edu}
\date{}
\maketitle

\begin{abstract}
Here we research the univariate quantitative approximation of real and
complex valued continuous functions on a compact interval or all the real
line by quasi-interpolation, Baskakov type and quadrature type neural
network operators. We perform also the related fractional approximation.
These approximations are derived by establishing Jackson type inequalities
involving the modulus of continuity of the engaged function or its high
order derivative or fractional derivatives. Our operators are defined by
using a density function induced by the error function. The approximations
are pointwise and with respect to the uniform norm. The related feed-forward
neural networks are with one hidden layer.
\end{abstract}



\section{Introduction}

The author in \cite{2} and \cite{3}, see Chapters 2-5, was the first to
establish neural network approximations to continuous functions with rates
by very specifically defined neural network operators of
Cardaliagnet-Euvrard and \textquotedblright Squashing\textquotedblright\
types, by employing the modulus of continuity of the engaged function or its
high order derivative, and producing very tight Jackson type inequalities.
He treats there both the univariate and multivariate cases. The defining
these operators \textquotedblright bell-shaped\textquotedblright\ and
\textquotedblright squashing\textquotedblright\ functions are assumed to be
of compact support. Also in \cite{3} he gives the $N$th order asymptotic
expansion for the error of weak approximation of these two operators to a
special natural class of smooth functions, see Chapters 4-5 there.

The author inspired by \cite{15}, continued his studies on neural networks
approximation by introducing and using the proper quasi-interpolation
operators of sigmoidal and hyperbolic tangent type which resulted into \cite%
{7}, \cite{9}, \cite{10}, \cite{11}, \cite{12}, by treating both the
univariate and multivariate cases. He did also the corresponding fractional
case \cite{13}.

The author here performs univariate error function based neural network
approximations to continuous functions over compact intervals of the real
line or over the whole $\mathbb{R}$, the he extends his results to complex
valued functions. Finally he treats completely the related fractional
approximation. All convergences here are with rates expressed via the
modulus of continuity of the involved function or its high order derivative,
or fractional derivatives and given by very tight Jackson type inequalities.

The author comes up with the ''right'' precisely defined
quasi-interpolation, Baskakov type and quadrature neural networks operators,
associated with the error function and related to a compact interval or real
line. Our compact intervals are not necessarily symmetric to the origin.
Some of our upper bounds to error quantity are very flexible and general. In
preparation to prove our results we establish important properties of the
basic density function defining our operators.

Feed-forward neural networks (FNNs) with one hidden layer, the only type of
networks we deal with in this article, are mathematically expressed as 
\begin{equation*}
N_{n}\left( x\right) =\sum_{j=0}^{n}c_{j}\sigma \left( \left\langle
a_{j}\cdot x\right\rangle +b_{j}\right) ,\text{ \ \ \ }x\in \mathbb{R}^{s}%
\text{, \ \ }s\in \mathbb{N}\text{,}
\end{equation*}%
where for $0\leq j\leq n$, $b_{j}\in \mathbb{R}$ are the thresholds, $%
a_{j}\in \mathbb{R}^{s}$ are the connection weights, $c_{j}\in \mathbb{R}$
are the coefficients, $\left\langle a_{j}\cdot x\right\rangle $ is the inner
product of $a_{j}$ and $x$, and $\sigma $ is the activation function of the
network. In many fundamental neural network models, the activation function
is the error. About neural networks in general read \cite{19}, \cite{20}, 
\cite{21}.

\section{Basics}

We consider here the (Gauss) error special function (\cite{1}, \cite{14}) 
\begin{equation}
\func{erf}\left( x\right) =\frac{2}{\sqrt{\pi }}\int_{0}^{x}e^{-t^{2}}dt,%
\text{ \ \ }x\in \mathbb{R},  \tag{1}  \label{r1}
\end{equation}%
which is a sigmoidal type function and a strictly increasing function.

It has the basic properties 
\begin{equation}
\func{erf}\left( 0\right) =0\text{, \ }\func{erf}\left( -x\right) =-\func{erf%
}\left( x\right) ,\text{ \ \ }\func{erf}\left( +\infty \right) =1\text{, \ \ 
}\func{erf}\left( -\infty \right) =-1,  \tag{2}  \label{r2}
\end{equation}%
and 
\begin{equation}
\left( \func{erf}\left( x\right) \right) ^{\prime }=\frac{2}{\sqrt{\pi }}%
e^{-x^{2}},\text{ \ \ }x\in \mathbb{R}\text{,}  \tag{3}  \label{r3}
\end{equation}

\begin{equation}
\int \func{erf}\left( x\right) dx=x\func{erf}\left( x\right) +\frac{%
e^{-x^{2}}}{\sqrt{\pi }}+C,  \tag{4}  \label{r4}
\end{equation}%
where $C$ is a constant.

The error function is related to the cumulative probability distribution
function of the standard normal distribution 
\begin{equation*}
\Phi \left( x\right) =\frac{1}{2}+\frac{1}{2}\func{erf}\left( \frac{x}{\sqrt{%
2}}\right) .
\end{equation*}%
We consider the activation function 
\begin{equation}
\chi \left( x\right) =\frac{1}{4}\left( \func{erf}\left( x+1\right) -\func{%
erf}\left( x-1\right) \right) ,\text{ \ }x\in \mathbb{R},  \tag{5}  \label{5}
\end{equation}%
and we notice that 
\begin{gather}
\chi \left( -x\right) =\frac{1}{4}\left( \func{erf}\left( -x+1\right) -\func{%
erf}\left( -x-1\right) \right) =  \notag \\
\frac{1}{4}\left( \func{erf}\left( -\left( x-1\right) \right) -\func{erf}%
\left( -\left( x+1\right) \right) \right) =\frac{1}{4}\left( -\func{erf}%
\left( x-1\right) +\func{erf}\left( x+1\right) \right) =\chi \left( x\right)
,  \tag{6}  \label{r6}
\end{gather}%
thus $\chi $ is an even function.

Since $x+1>x-1$, then $\func{erf}\left( x+1\right) >\func{erf}\left(
x-1\right) $, and $\chi \left( x\right) >0$, all $x\in \mathbb{R}$.

We see that 
\begin{equation}
\chi \left( 0\right) =\frac{\func{erf}\left( 1\right) }{2}\simeq \frac{0.843%
}{2}=0.4215.  \tag{7}  \label{r7}
\end{equation}%
Let $x>0,$ we have 
\begin{gather}
\chi ^{\prime }\left( x\right) =\frac{1}{4}\left( \frac{2}{\sqrt{\pi }}%
e^{-\left( x+1\right) ^{2}}-\frac{2}{\sqrt{\pi }}e^{-\left( x-1\right)
^{2}}\right) =  \notag \\
\frac{1}{2\sqrt{\pi }}\left( \frac{1}{e^{\left( x+1\right) ^{2}}}-\frac{1}{%
e^{\left( x-1\right) ^{2}}}\right) =\frac{1}{2\sqrt{\pi }}\left( \frac{%
e^{\left( x-1\right) ^{2}}-e^{\left( x+1\right) ^{2}}}{e^{\left( x+1\right)
^{2}}e^{\left( x-1\right) ^{2}}}\right) <0,  \tag{8}  \label{r8}
\end{gather}%
proving $\chi ^{\prime }\left( x\right) <0$, for $x>0$.

That is $\chi $ is strictly decreasing on $[0,\infty )$ and is strictly
increasing on $(-\infty ,0]$, and $\chi ^{\prime }\left( 0\right) =0$.

Clearly the $x$-axis is the horizontal asymptote on $\chi .$

Conclusion, $\chi $ is a bell symmetric function with maximum $\chi \left(
0\right) \simeq 0.4215.$

We further present

\begin{theorem}
\label{t1}We have that 
\begin{equation}
\sum_{i=-\infty }^{\infty }\chi \left( x-i\right) =1\text{, \ all }x\in 
\mathbb{R}.  \tag{9}  \label{r9}
\end{equation}
\end{theorem}

\begin{proof}
We notice 
\begin{equation*}
\sum_{i=-\infty }^{\infty }\func{erf}\left( x-i\right) -\func{erf}\left(
x-1-i\right) =
\end{equation*}%
\begin{equation}
\sum_{i=0}^{\infty }\left( \func{erf}\left( x-i\right) -\func{erf}\left(
x-1-i\right) \right) +\sum_{i=-\infty }^{-1}\left( \func{erf}\left(
x-i\right) -\func{erf}\left( x-1-i\right) \right) .  \tag{10}  \label{r10}
\end{equation}%
Furthermore ($\lambda \in \mathbb{Z}^{+}$) (telescoping sum) 
\begin{equation*}
\sum_{i=0}^{\infty }\left( \func{erf}\left( x-i\right) -\func{erf}\left(
x-1-i\right) \right) =
\end{equation*}%
\begin{equation*}
\underset{\lambda \rightarrow \infty }{\lim }\sum_{i=0}^{\lambda }\left( 
\func{erf}\left( x-i\right) -\func{erf}\left( x-1-i\right) \right) =
\end{equation*}
\begin{equation}
\func{erf}\left( x\right) -\underset{\lambda \rightarrow \infty }{\lim }%
\func{erf}\left( x-1-\lambda \right) =1+\func{erf}\left( x\right) .  \tag{11}
\label{r11}
\end{equation}%
Similarly we get 
\begin{equation*}
\sum_{i=-\infty }^{-1}\left( \func{erf}\left( x-i\right) -\func{erf}\left(
x-1-i\right) \right) =
\end{equation*}%
\begin{equation*}
\underset{\lambda \rightarrow \infty }{\lim }\sum_{i=-\lambda }^{-1}\left( 
\func{erf}\left( x-i\right) -\func{erf}\left( x-1-i\right) \right) =
\end{equation*}%
\begin{equation}
\underset{\lambda \rightarrow \infty }{\lim }\left( \func{erf}\left(
x+\lambda \right) -\func{erf}\left( x\right) \right) =1-\func{erf}\left(
x\right) .  \tag{12}  \label{r12}
\end{equation}%
Adding (\ref{r11}) and (\ref{r12}), we get 
\begin{equation}
\sum_{i=-\infty }^{\infty }\left( \func{erf}\left( x-i\right) -\func{erf}%
\left( x-1-i\right) \right) =2,\text{ \ \ for any }x\in \mathbb{R}\text{.} 
\tag{13}  \label{13r}
\end{equation}%
Hence (\ref{13r}) is true for $\left( x+1\right) $, giving us 
\begin{equation}
\sum_{i=-\infty }^{\infty }\left( \func{erf}\left( x+1-i\right) -\func{erf}%
\left( x-i\right) \right) =2,\text{ \ \ for any }x\in \mathbb{R}\text{.} 
\tag{14}  \label{r14}
\end{equation}%
Adding (\ref{13r}) and (\ref{r14}) we obtain%
\begin{equation}
\sum_{i=-\infty }^{\infty }\left( \func{erf}\left( x+1-i\right) -\func{erf}%
\left( x-1-i\right) \right) =4,\text{ \ \ for any }x\in \mathbb{R}\text{,} 
\tag{15}  \label{r15}
\end{equation}%
proving (\ref{r9}).
\end{proof}

Thus \ 
\begin{equation}
\sum_{i=-\infty }^{\infty }\chi \left( nx-i\right) =1,\ \ \ \ \forall \text{ 
}n\in \mathbb{N},\text{ }\forall \text{ }x\in \mathbb{R}.  \tag{16}
\label{r16}
\end{equation}

Furthermore we get:

Since $\chi $ is even it holds $\sum_{i=-\infty }^{\infty }\chi \left(
i-x\right) =1,\ $for any $x\in \mathbb{R}.$

Hence $\sum_{i=-\infty }^{\infty }\chi \left( i+x\right) =1$, \ $\forall $ $%
x\in \mathbb{R}$, and $\sum_{i=-\infty }^{\infty }\chi \left( x+i\right) =1$%
, \ $\forall $ $x\in \mathbb{R}$.

\begin{theorem}
\label{t2}It holds 
\begin{equation}
\int_{-\infty }^{\infty }\chi \left( x\right) dx=1.  \tag{17}  \label{r17}
\end{equation}
\end{theorem}

\begin{proof}
We notice that 
\begin{equation*}
\int_{-\infty }^{\infty }\chi \left( x\right) dx=\sum_{j=-\infty }^{\infty
}\int_{j}^{j+1}\chi \left( x\right) dx=\sum_{j=-\infty }^{\infty
}\int_{0}^{1}\chi \left( x+j\right) dx=
\end{equation*}%
\begin{equation*}
\int_{0}^{1}\left( \sum_{j=-\infty }^{\infty }\chi \left( x+j\right) \right)
dx=\int_{0}^{1}1dx=1.
\end{equation*}
\end{proof}

So $\chi \left( x\right) $ is a density function on $\mathbb{R}$.

\begin{theorem}
\label{t3}Let $0<\alpha <1,$ and $n\in \mathbb{N}$ with $n^{1-\alpha }\geq 3$%
. It holds 
\begin{equation}
\sum_{\left\{ 
\begin{array}{c}
k=-\infty \\ 
:\left| nx-k\right| \geq n^{1-\alpha }%
\end{array}%
\right. }^{\infty }\chi \left( nx-k\right) <\frac{1}{2\sqrt{\pi }\left(
n^{1-\alpha }-2\right) e^{\left( n^{1-\alpha }-2\right) ^{2}}}.  \tag{18}
\label{r18}
\end{equation}
\end{theorem}

\begin{proof}
Let $x\geq 1$. That is $0\leq x-1<x+1.$ Applying the mean value theorem we
get 
\begin{equation}
\chi \left( x\right) =\frac{1}{4}\left( \func{erf}\left( x+1\right) -\func{%
erf}\left( x-1\right) \right) =\frac{1}{\sqrt{\pi }}e^{-\xi ^{2}},  \tag{19}
\label{r19}
\end{equation}%
where $x-1<\xi <x+1.$

Hence 
\begin{equation}
\chi \left( x\right) <\frac{e^{-\left( x-1\right) ^{2}}}{\sqrt{\pi }},\text{
\ \ }x\geq 1.  \tag{20}  \label{r20}
\end{equation}%
Thus we have%
\begin{equation*}
\sum_{\left\{ 
\begin{array}{c}
k=-\infty \\ 
:\left| nx-k\right| \geq n^{1-\alpha }%
\end{array}%
\right. }^{\infty }\chi \left( nx-k\right) =\sum_{\left\{ 
\begin{array}{c}
k=-\infty \\ 
:\left| nx-k\right| \geq n^{1-\alpha }%
\end{array}%
\right. }^{\infty }\chi \left( \left| nx-k\right| \right) <
\end{equation*}%
\begin{equation}
\frac{1}{\sqrt{\pi }}\sum_{\left\{ 
\begin{array}{c}
k=-\infty \\ 
:\left| nx-k\right| \geq n^{1-\alpha }%
\end{array}%
\right. }^{\infty }e^{-\left( \left| nx-k\right| -1\right) ^{2}}\leq \frac{1%
}{\sqrt{\pi }}\int_{\left( n^{1-\alpha }-1\right) }^{\infty }e^{-\left(
x-1\right) ^{2}}dx  \tag{21}  \label{r21}
\end{equation}%
\begin{equation*}
=\frac{1}{\sqrt{\pi }}\int_{n^{1-\alpha }-2}^{\infty }e^{-z^{2}}dz
\end{equation*}%
(see section 3.7.3 of \cite{22})%
\begin{equation*}
=\frac{1}{2\sqrt{\pi }}\left( \min \left( \sqrt{\pi },\frac{1}{\left(
n^{1-\alpha }-2\right) }\right) \right) e^{-\left( n^{1-\alpha }-2\right)
^{2}}
\end{equation*}%
(by $n^{1-\alpha }-2\geq 1$, hence $\frac{1}{n^{1-\alpha }-2}\leq 1<\sqrt{%
\pi }$) 
\begin{equation}
<\frac{1}{2\sqrt{\pi }\left( n^{1-\alpha }-2\right) e^{\left( n^{1-\alpha
}-2\right) ^{2}}},  \tag{22}  \label{r22}
\end{equation}%
proving the claim.
\end{proof}

Denote by $\left\lfloor \cdot \right\rfloor $ the integral part of the
number and by $\left\lceil \cdot \right\rceil $ the ceiling of the number.

\begin{theorem}
\label{t4}Let $x\in \left[ a,b\right] \subset \mathbb{R}$ and $n\in \mathbb{N%
}$ so that $\left\lceil na\right\rceil \leq \left\lfloor nb\right\rfloor $.
It holds 
\begin{equation}
\frac{1}{\sum_{k=\left\lceil na\right\rceil }^{\left\lfloor nb\right\rfloor
}\chi \left( nx-k\right) }<\frac{1}{\chi \left( 1\right) }\simeq 4.019,\text{
\ }\forall \text{ }x\in \left[ a,b\right] .  \tag{23}  \label{r23}
\end{equation}
\end{theorem}

\begin{proof}
Let $x\in \left[ a,b\right] .$ We see that 
\begin{equation}
1=\sum_{k=-\infty }^{\infty }\chi \left( nx-k\right) >\sum_{k=\left\lceil
na\right\rceil }^{\left\lfloor nb\right\rfloor }\chi \left( nx-k\right) = 
\tag{24}  \label{r24}
\end{equation}%
\begin{equation*}
\sum_{k=\left\lceil na\right\rceil }^{\left\lfloor nb\right\rfloor }\chi
\left( \left| nx-k\right| \right) >\chi \left( \left| nx-k_{0}\right|
\right) ,
\end{equation*}%
$\forall $ $k_{0}\in \left[ \left\lceil na\right\rceil ,\left\lfloor
nb\right\rfloor \right] \cap \mathbb{Z}$.

We can choose $k_{0}\in \left[ \left\lceil na\right\rceil ,\left\lfloor
nb\right\rfloor \right] \cap \mathbb{Z}$ such that $\left| nx-k_{0}\right|
<1.$

Therefore 
\begin{equation*}
\chi \left( \left| nx-k_{0}\right| \right) >\chi \left( 1\right) =\frac{1}{4}%
\left( \func{erf}\left( 2\right) -\func{erf}\left( 0\right) \right) =
\end{equation*}%
\begin{equation}
\frac{\func{erf}\left( 2\right) }{4}=\frac{0.99533}{4}=0.2488325.  \tag{25}
\label{r25}
\end{equation}

Consequently we get 
\begin{equation}
\sum_{k=\left\lceil na\right\rceil }^{\left\lfloor nb\right\rfloor }\chi
\left( \left| nx-k\right| \right) >\chi \left( 1\right) \simeq 0.2488325, 
\tag{26}  \label{r26}
\end{equation}%
and 
\begin{equation}
\frac{1}{\sum_{k=\left\lceil na\right\rceil }^{\left\lfloor nb\right\rfloor
}\chi \left( \left| nx-k\right| \right) }<\frac{1}{\chi \left( 1\right) }%
\simeq 4.019,  \tag{27}  \label{r27}
\end{equation}%
proving the claim.
\end{proof}

\begin{remark}
\label{r5}We also notice that 
\begin{equation*}
1-\sum_{k=\left\lceil na\right\rceil }^{\left\lfloor nb\right\rfloor }\chi
\left( nb-k\right) =\sum_{k=-\infty }^{\left\lceil na\right\rceil -1}\chi
\left( nb-k\right) +\sum_{k=\left\lfloor nb\right\rfloor +1}^{\infty }\chi
\left( nb-k\right)
\end{equation*}%
\begin{equation*}
>\chi \left( nb-\left\lfloor nb\right\rfloor -1\right)
\end{equation*}%
(call $\varepsilon :=nb-\left\lfloor nb\right\rfloor $, $0\leq \varepsilon
<1 $) 
\begin{equation}
=\chi \left( \varepsilon -1\right) =\chi \left( 1-\varepsilon \right) \geq
\chi \left( 1\right) >0.  \tag{28}  \label{r28}
\end{equation}%
Therefore 
\begin{equation*}
\underset{n\rightarrow \infty }{\lim }\left( 1-\sum_{k=\left\lceil
na\right\rceil }^{\left\lfloor nb\right\rfloor }\chi \left( nb-k\right)
\right) >0.
\end{equation*}%
Similarly, 
\begin{equation*}
1-\sum_{k=\left\lceil na\right\rceil }^{\left\lfloor nb\right\rfloor }\chi
\left( na-k\right) =\sum_{k=-\infty }^{\left\lceil na\right\rceil -1}\chi
\left( na-k\right) +\sum_{k=\left\lfloor nb\right\rfloor +1}^{\infty }\chi
\left( na-k\right)
\end{equation*}%
\begin{equation*}
>\chi \left( na-\left\lceil na\right\rceil +1\right)
\end{equation*}%
(call $\eta :=\left\lceil na\right\rceil -na$, \ $0\leq \eta <1$) 
\begin{equation*}
=\chi \left( 1-\eta \right) \geq \chi \left( 1\right) >0.
\end{equation*}%
Therefore again 
\begin{equation}
\underset{n\rightarrow \infty }{\lim }\left( 1-\sum_{k=\left\lceil
na\right\rceil }^{\left\lfloor nb\right\rfloor }\chi \left( na-k\right)
\right) >0.  \tag{29}  \label{r29}
\end{equation}%
Hence we derive that 
\begin{equation}
\underset{n\rightarrow \infty }{\lim }\sum_{k=\left\lceil na\right\rceil
}^{\left\lfloor nb\right\rfloor }\chi \left( nx-k\right) \neq 1,  \tag{30}
\label{r30}
\end{equation}%
for at least some $x\in \left[ a,b\right] $.
\end{remark}

\begin{note}
\label{n6}For large enough $n$ we always obtain $\left\lceil na\right\rceil
\leq \left\lfloor nb\right\rfloor $. Also $a\leq \frac{k}{n}\leq b$, iff $%
\left\lceil na\right\rceil \leq k\leq \left\lfloor nb\right\rfloor $. In
general it holds (by $(\ref{r16})$) that 
\begin{equation}
\sum_{k=\left\lceil na\right\rceil }^{\left\lfloor nb\right\rfloor }\chi
\left( nx-k\right) \leq 1.  \tag{31}  \label{r31}
\end{equation}
\end{note}

We give

\begin{definition}
\label{d7}Let $f\in C\left( \left[ a,b\right] \right) $ $n\in \mathbb{N}$.
We set 
\begin{equation}
A_{n}\left( f,x\right) =\frac{\sum_{k=\left\lceil na\right\rceil
}^{\left\lfloor nb\right\rfloor }f\left( \frac{k}{n}\right) \chi \left(
nx-k\right) }{\sum_{k=\left\lceil na\right\rceil }^{\left\lfloor
nb\right\rfloor }\chi \left( nx-k\right) }\text{, \ }\forall \text{\ }x\in %
\left[ a.b\right] ,  \tag{32}  \label{r32}
\end{equation}%
$A_{n}$ is a neural network operator.
\end{definition}

\begin{definition}
\label{d8}Let $f\in C_{B}\left( \mathbb{R}\right) ,$ (continuous and bounded
functions on $\mathbb{R}),$ $n\in \mathbb{N}.$\ We introduce the
quasi-interpolation operator%
\begin{equation}
B_{n}\left( f,x\right) :=\sum_{k=-\infty }^{\infty }f\left( \frac{k}{n}%
\right) \chi \left( nx-k\right) ,\text{ }\ \forall \text{ }x\in \mathbb{R}, 
\tag{33}  \label{r33}
\end{equation}%
and the Kantorovich type operator 
\begin{equation}
C_{n}\left( f,x\right) =\sum_{k=-\infty }^{\infty }\left( n\int_{\frac{k}{n}%
}^{\frac{k+1}{n}}f\left( t\right) dt\right) \chi \left( nx-k\right) ,\text{ }%
\ \forall \text{ }x\in \mathbb{R}\text{.}  \tag{34}  \label{r34}
\end{equation}%
$B_{n}$, $C_{n}$ are neural network operators.
\end{definition}

Also we give

\begin{definition}
\label{d9}Let $f\in C_{B}\left( \mathbb{R}\right) ,$ $n\in \mathbb{N}.$\ Let 
$\theta \in \mathbb{N}$, $w_{r}\geq 0$, $\sum_{r=0}^{\theta }w_{r}=1$, $k\in 
\mathbb{Z}$, and%
\begin{equation}
\delta _{nk}\left( f\right) =\sum_{r=0}^{\theta }w_{r}f\left( \frac{k}{n}+%
\frac{r}{n\theta }\right) .  \tag{35}  \label{r35}
\end{equation}%
We put 
\begin{equation}
D_{n}\left( f,x\right) =\sum_{k=-\infty }^{\infty }\delta _{nk}\left(
f\right) \chi \left( nx-k\right) ,\text{ }\ \forall \text{ }x\in \mathbb{R}%
\text{.}  \tag{36}  \label{r36}
\end{equation}%
$D_{n}$ is a neural network operator of quadrature type.
\end{definition}

We need

\begin{definition}
\label{d10}For $f\in C\left( \left[ a,b\right] \right) $, the first modulus
of continuity is given by 
\begin{equation}
\omega _{1}\left( f,\delta \right) :=\underset{%
\begin{array}{c}
x,y\in \left[ a,b\right] \\ 
\left| x-y\right| \leq \delta%
\end{array}%
}{\sup }\left| f\left( x\right) -f\left( y\right) \right| \text{, \ }\delta
>0.  \tag{37}  \label{r37}
\end{equation}%
We have that $\underset{\delta \rightarrow 0}{\lim }\omega _{1}\left(
f,\delta \right) =0.$

Similarly $\omega _{1}\left( f,\delta \right) $ is defined for $f\in
C_{B}\left( \mathbb{R}\right) $.
\end{definition}

We know that, $f$ is uniformly continuous on $\mathbb{R}$ iff $\underset{%
\delta \rightarrow 0}{\lim }\omega _{1}\left( f,\delta \right) =0.$

We make

\begin{remark}
\label{re11}We notice the following, that 
\begin{equation}
A_{n}\left( f,x\right) -f\left( x\right) \overset{(\ref{r32})}{=}\frac{%
\sum_{k=\left\lceil na\right\rceil }^{\left\lfloor nb\right\rfloor }f\left( 
\frac{k}{n}\right) \chi \left( nx-k\right) -f\left( x\right)
\sum_{k=\left\lceil na\right\rceil }^{\left\lfloor nb\right\rfloor }\chi
\left( nx-k\right) }{\sum_{k=\left\lceil na\right\rceil }^{\left\lfloor
nb\right\rfloor }\chi \left( nx-k\right) },  \tag{38}  \label{r38}
\end{equation}%
using (\ref{r23}) we get, 
\begin{equation}
\left| A_{n}\left( f,x\right) -f\left( x\right) \right| \leq \left(
4.019\right) \left| \sum_{k=\left\lceil na\right\rceil }^{\left\lfloor
nb\right\rfloor }f\left( \frac{k}{n}\right) \chi \left( nx-k\right) -f\left(
x\right) \sum_{k=\left\lceil na\right\rceil }^{\left\lfloor nb\right\rfloor
}\chi \left( nx-k\right) \right| .  \tag{39}  \label{r39}
\end{equation}%
Again here $0<\alpha <1$ and $n\in \mathbb{N}$ with $n^{1-\alpha }\geq 3$.
Let the fixed $K,L>0;$ for the linear combination $\frac{K}{n^{\alpha }}+%
\frac{L}{\left( n^{1-\alpha }-2\right) e^{\left( n^{1-\alpha }-2\right) ^{2}}%
},$ the dominant rate of convergence to zero, as $n\rightarrow \infty $, is $%
n^{-\alpha }.$ The closer $\alpha $ is to $1$, we get faster and better rate
of convergence to zero.
\end{remark}

In this article we study basic approximation properties of $%
A_{n},B_{n},C_{n},D_{n}$ neural network operators. That is, the quantitative
pointwise and uniform convergence of these operators to the unit operator $I$%
.

\section{Real Neural Network Approximations}

Here we present a series of neural network approximations to a function
given with rates.

We give

\begin{theorem}
\label{t12}Let $f\in C\left( \left[ a,b\right] \right) ,$ $0<\alpha <1$, $%
x\in \left[ a,b\right] ,$ $n\in \mathbb{N}$ with $n^{1-\alpha }\geq 3$, $%
\left\| \cdot \right\| _{\infty }$ is the supremum norm. Then

1) 
\begin{equation}
\left| A_{n}\left( f,x\right) -f\left( x\right) \right| \leq \left(
4.019\right) \left[ \omega _{1}\left( f,\frac{1}{n^{\alpha }}\right) +\frac{%
\left\| f\right\| _{\infty }}{\sqrt{\pi }\left( n^{1-\alpha }-2\right)
e^{\left( n^{1-\alpha }-2\right) ^{2}}}\right] =:\mu _{1n},  \tag{40}
\label{r40}
\end{equation}

2) 
\begin{equation}
\left\| A_{n}\left( f\right) -f\right\| _{\infty }\leq \mu _{1n}.  \tag{41}
\label{r41}
\end{equation}%
We notice that $\underset{n\rightarrow \infty }{\lim }A_{n}\left( f\right)
=f $, pointwise and uniformly.
\end{theorem}

\begin{proof}
Using (\ref{r39}) we get 
\begin{equation*}
\left| A_{n}\left( f,x\right) -f\left( x\right) \right| \leq \left(
4.019\right) \left[ \sum_{k=\left\lceil na\right\rceil }^{\left\lfloor
nb\right\rfloor }\left| f\left( \frac{k}{n}\right) -f\left( x\right) \right|
\chi \left( nx-k\right) \right] \leq
\end{equation*}%
\begin{equation*}
\left( 4.019\right) \left[ \sum_{\left\{ 
\begin{array}{c}
k=\left\lceil na\right\rceil \\ 
\left| \frac{k}{n}-x\right| \leq \frac{1}{n^{\alpha }}%
\end{array}%
\right. }^{\left\lfloor nb\right\rfloor }\left| f\left( \frac{k}{n}\right)
-f\left( x\right) \right| \chi \left( nx-k\right) +\right.
\end{equation*}%
\begin{equation}
\left. \sum_{\left\{ 
\begin{array}{c}
k=\left\lceil na\right\rceil \\ 
\left| \frac{k}{n}-x\right| >\frac{1}{n^{\alpha }}%
\end{array}%
\right. }^{\left\lfloor nb\right\rfloor }\left| f\left( \frac{k}{n}\right)
-f\left( x\right) \right| \chi \left( nx-k\right) \right] \leq  \tag{42}
\label{r42}
\end{equation}%
\begin{equation*}
\left( 4.019\right) \left[ \omega _{1}\left( f,\frac{1}{n^{\alpha }}\right)
\left( \sum_{k=\left\lceil na\right\rceil }^{\left\lfloor nb\right\rfloor
}\chi \left( nx-k\right) \right) +\right.
\end{equation*}%
\begin{equation}
\left. 2\left\| f\right\| _{\infty }\left( \sum_{\left\{ 
\begin{array}{c}
k=\left\lceil na\right\rceil \\ 
\left| nx-k\right| \geq n^{1-\alpha }%
\end{array}%
\right. }^{\left\lfloor nb\right\rfloor }\chi \left( nx-k\right) \right) %
\right] \overset{\text{(by (\ref{r18}), (\ref{r31}))}}{\leq }  \tag{43}
\label{r43}
\end{equation}%
\begin{equation*}
\left( 4.019\right) \left[ \omega _{1}\left( f,\frac{1}{n^{\alpha }}\right) +%
\frac{\left\| f\right\| _{\infty }}{\sqrt{\pi }\left( n^{1-\alpha }-2\right)
e^{\left( n^{1-\alpha }-2\right) ^{2}}}\right] ,
\end{equation*}%
proving the claim.
\end{proof}

We continue with

\begin{theorem}
\label{t13}Let $f\in C_{B}\left( \mathbb{R}\right) ,$ $0<\alpha <1$, $x\in 
\mathbb{R},$ $n\in \mathbb{N}$ with $n^{1-\alpha }\geq 3$. Then

1) 
\begin{equation}
\left| B_{n}\left( f,x\right) -f\left( x\right) \right| \leq \omega
_{1}\left( f,\frac{1}{n^{\alpha }}\right) +\frac{\left\| f\right\| _{\infty }%
}{\sqrt{\pi }\left( n^{1-\alpha }-2\right) e^{\left( n^{1-\alpha }-2\right)
^{2}}}=:\mu _{2n},  \tag{44}  \label{r44}
\end{equation}

2) 
\begin{equation}
\left\| B_{n}\left( f\right) -f\right\| _{\infty }\leq \mu _{2n}.  \tag{45}
\label{r45}
\end{equation}%
For $f\in \left( C_{B}\left( \mathbb{R}\right) \cap C_{u}\left( \mathbb{R}%
\right) \right) $ ($C_{u}\left( \mathbb{R}\right) $ uniformly continuous
functions on $\mathbb{R}$) we get $\underset{n\rightarrow \infty }{\lim }%
B_{n}\left( f\right) =f$, pointwise and uniformly.
\end{theorem}

\begin{proof}
We see that%
\begin{equation*}
\left\vert B_{n}\left( f,x\right) -f\left( x\right) \right\vert \overset{%
\text{(by (\ref{r16}), (\ref{r33}))}}{=}\left\vert \sum_{k=-\infty }^{\infty
}\left( f\left( \frac{k}{n}\right) -f\left( x\right) \right) \chi \left(
nx-k\right) \right\vert \leq 
\end{equation*}%
\begin{equation*}
\sum_{k=-\infty }^{\infty }\left\vert f\left( \frac{k}{n}\right) -f\left(
x\right) \right\vert \chi \left( nx-k\right) \leq 
\end{equation*}%
\begin{equation*}
\sum_{\left\{ 
\begin{array}{c}
k=-\infty  \\ 
\left\vert \frac{k}{n}-x\right\vert \leq \frac{1}{n^{\alpha }}%
\end{array}%
\right. }^{\infty }\left\vert f\left( \frac{k}{n}\right) -f\left( x\right)
\right\vert \chi \left( nx-k\right) +
\end{equation*}%
\begin{equation}
\sum_{\left\{ 
\begin{array}{c}
k=-\infty  \\ 
\left\vert \frac{k}{n}-x\right\vert \geq \frac{1}{n^{\alpha }}%
\end{array}%
\right. }^{\infty }\left\vert f\left( \frac{k}{n}\right) -f\left( x\right)
\right\vert \chi \left( nx-k\right) \leq   \tag{46}  \label{r46}
\end{equation}%
\begin{equation*}
\omega _{1}\left( f,\frac{1}{n^{\alpha }}\right) \left( \sum_{\left\{ 
\begin{array}{c}
k=-\infty  \\ 
\left\vert \frac{k}{n}-x\right\vert \leq \frac{1}{n^{\alpha }}%
\end{array}%
\right. }^{\infty }\chi \left( nx-k\right) \right) +
\end{equation*}%
\begin{equation}
2\left\Vert f\right\Vert _{\infty }\left( \sum_{\left\{ 
\begin{array}{c}
k=-\infty  \\ 
\left\vert nx-k\right\vert \geq n^{1-\alpha }%
\end{array}%
\right. }^{\infty }\chi \left( nx-k\right) \right) \overset{\text{(by (\ref%
{r16}), (\ref{r18}))}}{\leq }  \tag{47}  \label{r47}
\end{equation}%
\begin{equation*}
\omega _{1}\left( f,\frac{1}{n^{\alpha }}\right) +\frac{\left\Vert
f\right\Vert _{\infty }}{\sqrt{\pi }\left( n^{1-\alpha }-2\right) e^{\left(
n^{1-\alpha }-2\right) ^{2}}}.
\end{equation*}
\end{proof}

We continue with

\begin{theorem}
\label{tt14}Let $f\in C_{B}\left( \mathbb{R}\right) ,$ $0<\alpha <1$, $x\in 
\mathbb{R},$ $n\in \mathbb{N}$ with $n^{1-\alpha }\geq 3$. Then

1) 
\begin{equation}
\left| C_{n}\left( f,x\right) -f\left( x\right) \right| \leq \omega
_{1}\left( f,\frac{1}{n}+\frac{1}{n^{\alpha }}\right) +\frac{\left\|
f\right\| _{\infty }}{\sqrt{\pi }\left( n^{1-\alpha }-2\right) e^{\left(
n^{1-\alpha }-2\right) ^{2}}}=:\mu _{3n},  \tag{48}  \label{r48}
\end{equation}

2) 
\begin{equation}
\left\| C_{n}\left( f\right) -f\right\| _{\infty }\leq \mu _{3n}.  \tag{49}
\label{r49}
\end{equation}%
For $f\in \left( C_{B}\left( \mathbb{R}\right) \cap C_{u}\left( \mathbb{R}%
\right) \right) $ we get $\underset{n\rightarrow \infty }{\lim }C_{n}\left(
f\right) =f$, pointwise and uniformly.
\end{theorem}

\begin{proof}
We notice that 
\begin{equation}
\int_{\frac{k}{n}}^{\frac{k+1}{n}}f\left( t\right) dt=\int_{0}^{\frac{1}{n}%
}f\left( t+\frac{k}{n}\right) dt.  \tag{50}  \label{r50}
\end{equation}%
Hence we can write 
\begin{equation}
C_{n}\left( f,x\right) =\sum_{k=-\infty }^{\infty }\left( n\int_{0}^{\frac{1%
}{n}}f\left( t+\frac{k}{n}\right) dt\right) \chi \left( nx-k\right) . 
\tag{51}  \label{r51}
\end{equation}%
We observe that%
\begin{equation}
\left\vert C_{n}\left( f,x\right) -f\left( x\right) \right\vert =\left\vert
\sum_{k=-\infty }^{\infty }\left( \left( n\int_{0}^{\frac{1}{n}}f\left( t+%
\frac{k}{n}\right) dt\right) -f\left( x\right) \right) \chi \left(
nx-k\right) \right\vert =  \tag{52}  \label{r52}
\end{equation}%
\begin{equation*}
\left\vert \sum_{k=-\infty }^{\infty }\left( n\int_{0}^{\frac{1}{n}}\left(
f\left( t+\frac{k}{n}\right) -f\left( x\right) \right) dt\right) \chi \left(
nx-k\right) \right\vert \leq 
\end{equation*}%
\begin{equation*}
\sum_{k=-\infty }^{\infty }\left( n\int_{0}^{\frac{1}{n}}\left\vert f\left(
t+\frac{k}{n}\right) -f\left( x\right) \right\vert dt\right) \chi \left(
nx-k\right) \leq 
\end{equation*}%
\begin{equation}
\sum_{\left\{ 
\begin{array}{c}
k=-\infty  \\ 
\left\vert x-\frac{k}{n}\right\vert \leq \frac{1}{n^{\alpha }}%
\end{array}%
\right. }^{\infty }\left( n\int_{0}^{\frac{1}{n}}\left\vert f\left( t+\frac{k%
}{n}\right) -f\left( x\right) \right\vert dt\right) \chi \left( nx-k\right) +
\tag{53}  \label{r53}
\end{equation}%
\begin{equation*}
\sum_{\left\{ 
\begin{array}{c}
k=-\infty  \\ 
\left\vert x-\frac{k}{n}\right\vert \geq \frac{1}{n^{\alpha }}%
\end{array}%
\right. }^{\infty }\left( n\int_{0}^{\frac{1}{n}}\left\vert f\left( t+\frac{k%
}{n}\right) -f\left( x\right) \right\vert dt\right) \chi \left( nx-k\right)
\leq 
\end{equation*}%
\begin{equation*}
\sum_{\left\{ 
\begin{array}{c}
k=-\infty  \\ 
\left\vert x-\frac{k}{n}\right\vert \leq \frac{1}{n^{\alpha }}%
\end{array}%
\right. }^{\infty }\left( n\int_{0}^{\frac{1}{n}}\omega _{1}\left(
f,\left\vert t+\frac{k}{n}-x\right\vert \right) dt\right) \chi \left(
nx-k\right) +
\end{equation*}%
\begin{equation}
2\left\Vert f\right\Vert _{\infty }\left( \sum_{\left\{ 
\begin{array}{c}
k=-\infty  \\ 
\left\vert nx-k\right\vert \geq n^{1-\alpha }%
\end{array}%
\right. }^{\infty }\chi \left( \left\vert nx-k\right\vert \right) \right)
\leq   \tag{54}  \label{r54}
\end{equation}%
\begin{equation*}
\sum_{\left\{ 
\begin{array}{c}
k=-\infty  \\ 
\left\vert nx-k\right\vert \leq n^{1-\alpha }%
\end{array}%
\right. }^{\infty }\left( n\int_{0}^{\frac{1}{n}}\omega _{1}\left(
f,\left\vert t\right\vert +\frac{1}{n^{\alpha }}\right) dt\right) \chi
\left( nx-k\right) 
\end{equation*}%
\begin{equation*}
+\frac{\left\Vert f\right\Vert _{\infty }}{\sqrt{\pi }\left( n^{1-\alpha
}-2\right) e^{\left( n^{1-\alpha }-2\right) ^{2}}}\leq 
\end{equation*}%
\begin{equation}
\omega _{1}\left( f,\frac{1}{n}+\frac{1}{n^{\alpha }}\right) \left(
\sum_{\left\{ 
\begin{array}{c}
k=-\infty  \\ 
\left\vert nx-k\right\vert \leq n^{1-\alpha }%
\end{array}%
\right. }^{\infty }\chi \left( nx-k\right) \right)   \tag{55}  \label{r55}
\end{equation}%
\begin{equation*}
+\frac{\left\Vert f\right\Vert _{\infty }}{\sqrt{\pi }\left( n^{1-\alpha
}-2\right) e^{\left( n^{1-\alpha }-2\right) ^{2}}}\leq 
\end{equation*}%
\begin{equation*}
\omega _{1}\left( f,\frac{1}{n}+\frac{1}{n^{\alpha }}\right) +\frac{%
\left\Vert f\right\Vert _{\infty }}{\sqrt{\pi }\left( n^{1-\alpha }-2\right)
e^{\left( n^{1-\alpha }-2\right) ^{2}}},
\end{equation*}%
proving the claim.
\end{proof}

We give next

\begin{theorem}
\label{tt15}Let $f\in C_{B}\left( \mathbb{R}\right) ,$ $0<\alpha <1$, $x\in 
\mathbb{R},$ $n\in \mathbb{N}$ with $n^{1-\alpha }\geq 3$. Then

1) 
\begin{equation}
\left| D_{n}\left( f,x\right) -f\left( x\right) \right| \leq \mu _{3n}, 
\tag{56}  \label{r56}
\end{equation}%
and

2) 
\begin{equation}
\left\| D_{n}\left( f\right) -f\right\| _{\infty }\leq \mu _{3n},  \tag{57}
\label{r57}
\end{equation}%
where $\mu _{3n}$ as in (\ref{r48}).

For $f\in \left( C_{B}\left( \mathbb{R}\right) \cap C_{u}\left( \mathbb{R}%
\right) \right) $ we get $\underset{n\rightarrow \infty }{\lim }D_{n}\left(
f\right) =f$, pointwise and uniformly.
\end{theorem}

\begin{proof}
We see that%
\begin{equation*}
\left| D_{n}\left( f,x\right) -f\left( x\right) \right| \overset{\text{(by (%
\ref{r35}), (\ref{r36}))}}{=}
\end{equation*}%
\begin{equation*}
\left| \sum_{k=-\infty }^{\infty }\left( \left( \sum_{r=0}^{\theta
}w_{r}f\left( \frac{k}{n}+\frac{r}{n\theta }\right) \right) -f\left(
x\right) \right) \chi \left( nx-k\right) \right| =
\end{equation*}%
\begin{equation}
\left| \sum_{k=-\infty }^{\infty }\left( \sum_{r=0}^{\theta }w_{r}\left(
f\left( \frac{k}{n}+\frac{r}{n\theta }\right) -f\left( x\right) \right)
\right) \chi \left( nx-k\right) \right| \leq  \tag{58}  \label{r58}
\end{equation}%
\begin{equation*}
\sum_{k=-\infty }^{\infty }\left( \sum_{r=0}^{\theta }w_{r}\left| f\left( 
\frac{k}{n}+\frac{r}{n\theta }\right) -f\left( x\right) \right| \right) \chi
\left( nx-k\right) \leq
\end{equation*}%
\begin{equation}
\sum_{\left\{ 
\begin{array}{c}
k=-\infty \\ 
\left| \frac{k}{n}-x\right| \leq \frac{1}{n^{\alpha }}%
\end{array}%
\right. }^{\infty }\left( \sum_{r=0}^{\theta }w_{r}\left| f\left( \frac{k}{n}%
+\frac{r}{n\theta }\right) -f\left( x\right) \right| \right) \chi \left(
nx-k\right) +  \tag{59}  \label{r59}
\end{equation}%
\begin{equation*}
2\left\| f\right\| _{\infty }\sum_{\left\{ 
\begin{array}{c}
k=-\infty \\ 
\left| nx-k\right| \geq n^{1-\alpha }%
\end{array}%
\right. }^{\infty }\chi \left( \left| nx-k\right| \right) \leq
\end{equation*}%
(see that $\frac{r}{n\theta }\leq \frac{1}{n}$)%
\begin{equation*}
\sum_{\left\{ 
\begin{array}{c}
k=-\infty \\ 
\left| \frac{k}{n}-x\right| \leq \frac{1}{n^{\alpha }}%
\end{array}%
\right. }^{\infty }\left( \sum_{r=0}^{\theta }w_{r}\omega _{1}\left( f,\frac{%
1}{n^{\alpha }}+\frac{1}{n}\right) \right) \chi \left( nx-k\right) +
\end{equation*}%
\begin{equation}
\frac{\left\| f\right\| _{\infty }}{\sqrt{\pi }\left( n^{1-\alpha }-2\right)
e^{\left( n^{1-\alpha }-2\right) ^{2}}}\leq  \tag{60}  \label{r60}
\end{equation}%
\begin{equation*}
\omega _{1}\left( f,\frac{1}{n^{\alpha }}+\frac{1}{n}\right) +\frac{\left\|
f\right\| _{\infty }}{\sqrt{\pi }\left( n^{1-\alpha }-2\right) e^{\left(
n^{1-\alpha }-2\right) ^{2}}}=\mu _{3n},
\end{equation*}%
proving the claim.
\end{proof}

In the next we discuss high order of approximation by using the smoothness
of $f$.

\begin{theorem}
\label{t16}Let $f\in C^{N}\left( \left[ a,b\right] \right) $, $n,N\in 
\mathbb{N}$, $n^{1-\alpha }\geq 3,$ $0<\alpha <1$, $x\in \left[ a,b\right] $%
. Then

i) 
\begin{equation}
\left\vert A_{n}\left( f,x\right) -f\left( x\right) \right\vert \leq \left(
4.019\right) \cdot   \tag{61}  \label{r61}
\end{equation}%
\begin{equation*}
\left\{ \sum_{j=1}^{N}\frac{\left\vert f^{\left( j\right) }\left( x\right)
\right\vert }{j!}\left[ \frac{1}{n^{\alpha j}}+\frac{\left( b-a\right) ^{j}}{%
2\sqrt{\pi }\left( n^{1-\alpha }-2\right) e^{\left( n^{1-\alpha }-2\right)
^{2}}}\right] +\right. 
\end{equation*}%
\begin{equation*}
\left. \left[ \omega _{1}\left( f^{\left( N\right) },\frac{1}{n^{\alpha }}%
\right) \frac{1}{n^{\alpha N}N!}+\frac{\left\Vert f^{\left( N\right)
}\right\Vert _{\infty }\left( b-a\right) ^{N}}{N!\sqrt{\pi }\left(
n^{1-\alpha }-2\right) e^{\left( n^{1-\alpha }-2\right) ^{2}}}\right]
\right\} ,
\end{equation*}

ii) assume further $f^{\left( j\right) }\left( x_{0}\right) =0$, $j=1,...,N$%
, for some $x_{0}\in \left[ a,b\right] $, it holds 
\begin{equation}
\left| A_{n}\left( f,x_{0}\right) -f\left( x_{0}\right) \right| \leq \left(
4.019\right) \cdot  \tag{62}  \label{r62}
\end{equation}%
\begin{equation*}
\left[ \omega _{1}\left( f^{\left( N\right) },\frac{1}{n^{\alpha }}\right) 
\frac{1}{n^{\alpha N}N!}+\frac{\left\| f^{\left( N\right) }\right\| _{\infty
}\left( b-a\right) ^{N}}{N!\sqrt{\pi }\left( n^{1-\alpha }-2\right)
e^{\left( n^{1-\alpha }-2\right) ^{2}}}\right] ,
\end{equation*}%
notice here the extremely high rate of convergence at $n^{-\left( N+1\right)
\alpha },$

iii) 
\begin{equation}
\left\| A_{n}\left( f\right) -f\right\| _{\infty }\leq \left( 4.019\right)
\cdot  \tag{63}  \label{r63}
\end{equation}%
\begin{equation*}
\left\{ \sum_{j=1}^{N}\frac{\left\| f^{\left( j\right) }\right\| _{\infty }}{%
j!}\left[ \frac{1}{n^{\alpha j}}+\frac{\left( b-a\right) ^{j}}{2\sqrt{\pi }%
\left( n^{1-\alpha }-2\right) e^{\left( n^{1-\alpha }-2\right) ^{2}}}\right]
+\right.
\end{equation*}%
\begin{equation*}
\left. \left[ \omega _{1}\left( f^{\left( N\right) },\frac{1}{n^{\alpha }}%
\right) \frac{1}{n^{\alpha N}N!}+\frac{\left\| f^{\left( N\right) }\right\|
_{\infty }\left( b-a\right) ^{N}}{N!\sqrt{\pi }\left( n^{1-\alpha }-2\right)
e^{\left( n^{1-\alpha }-2\right) ^{2}}}\right] \right\} .
\end{equation*}
\end{theorem}

\begin{proof}
We use (\ref{r39}).

Call 
\begin{equation*}
A_{n}^{\ast }\left( f,x\right) :=\sum_{k=\left\lceil na\right\rceil
}^{\left\lfloor nb\right\rfloor }f\left( \frac{k}{n}\right) \chi \left(
nx-k\right) ,
\end{equation*}%
that is 
\begin{equation*}
A_{n}\left( f,x\right) =\frac{A_{n}^{\ast }\left( f,x\right) }{%
\sum_{k=\left\lceil na\right\rceil }^{\left\lfloor nb\right\rfloor }\chi
\left( nx-k\right) }.
\end{equation*}

Next we apply Taylor's formula with integral remainder.

We have (here $\frac{k}{n},x\in \left[ a,b\right] $) 
\begin{equation*}
f\left( \frac{k}{n}\right) =\sum_{j=0}^{N}\frac{f^{\left( j\right) }\left(
x\right) }{j!}\left( \frac{k}{n}-x\right) ^{j}+\int_{x}^{\frac{k}{n}}\left(
f^{\left( N\right) }\left( t\right) -f^{\left( N\right) }\left( x\right)
\right) \frac{\left( \frac{k}{n}-t\right) ^{N-1}}{\left( N-1\right) !}dt.
\end{equation*}%
Then 
\begin{equation*}
f\left( \frac{k}{n}\right) \chi \left( nx-k\right) =\sum_{j=0}^{N}\frac{%
f^{\left( j\right) }\left( x\right) }{j!}\chi \left( nx-k\right) \left( 
\frac{k}{n}-x\right) ^{j}+
\end{equation*}%
\begin{equation*}
\chi \left( nx-k\right) \int_{x}^{\frac{k}{n}}\left( f^{\left( N\right)
}\left( t\right) -f^{\left( N\right) }\left( x\right) \right) \frac{\left( 
\frac{k}{n}-t\right) ^{N-1}}{\left( N-1\right) !}dt.
\end{equation*}%
Hence 
\begin{equation*}
\sum_{k=\left\lceil na\right\rceil }^{\left\lfloor nb\right\rfloor }f\left( 
\frac{k}{n}\right) \chi \left( nx-k\right) -f\left( x\right)
\sum_{k=\left\lceil na\right\rceil }^{\left\lfloor nb\right\rfloor }\chi
\left( nx-k\right) =
\end{equation*}%
\begin{equation*}
\sum_{j=1}^{N}\frac{f^{\left( j\right) }\left( x\right) }{j!}%
\sum_{k=\left\lceil na\right\rceil }^{\left\lfloor nb\right\rfloor }\chi
\left( nx-k\right) \left( \frac{k}{n}-x\right) ^{j}+
\end{equation*}%
\begin{equation*}
\sum_{k=\left\lceil na\right\rceil }^{\left\lfloor nb\right\rfloor }\chi
\left( nx-k\right) \int_{x}^{\frac{k}{n}}\left( f^{\left( N\right) }\left(
t\right) -f^{\left( N\right) }\left( x\right) \right) \frac{\left( \frac{k}{n%
}-t\right) ^{N-1}}{\left( N-1\right) !}dt.
\end{equation*}%
Thus 
\begin{equation}
A_{n}^{\ast }\left( f,x\right) -f\left( x\right) \left( \sum_{k=\left\lceil
na\right\rceil }^{\left\lfloor nb\right\rfloor }\chi \left( nx-k\right)
\right) =\sum_{j=1}^{N}\frac{f^{\left( j\right) }\left( x\right) }{j!}%
A_{n}^{\ast }\left( \left( \cdot -x\right) ^{j}\right) +\Lambda _{n}\left(
x\right) ,  \tag{64}  \label{r64}
\end{equation}%
where 
\begin{equation}
\Lambda _{n}\left( x\right) :=\sum_{k=\left\lceil na\right\rceil
}^{\left\lfloor nb\right\rfloor }\chi \left( nx-k\right) \int_{x}^{\frac{k}{n%
}}\left( f^{\left( N\right) }\left( t\right) -f^{\left( N\right) }\left(
x\right) \right) \frac{\left( \frac{k}{n}-t\right) ^{N-1}}{\left( N-1\right)
!}dt.  \tag{65}  \label{r65}
\end{equation}%
We assume that $b-a>\frac{1}{n^{\alpha }}$, which is always the case for
large enough $n\in \mathbb{N}$, that is when $n>\left\lceil \left(
b-a\right) ^{-\frac{1}{\alpha }}\right\rceil .$

Thus $\left| \frac{k}{n}-x\right| \leq \frac{1}{n^{\alpha }}$ or $\left| 
\frac{k}{n}-x\right| >\frac{1}{n^{\alpha }}$.

As in \cite{3}, pp. 72-73 for 
\begin{equation}
\gamma :=\int_{x}^{\frac{k}{n}}\left( f^{\left( N\right) }\left( t\right)
-f^{\left( N\right) }\left( x\right) \right) \frac{\left( \frac{k}{n}%
-t\right) ^{N-1}}{\left( N-1\right) !}dt,  \tag{66}  \label{r66}
\end{equation}%
in case of $\left| \frac{k}{n}-x\right| \leq \frac{1}{n^{\alpha }}$, we find
that 
\begin{equation*}
\left| \gamma \right| \leq \omega _{1}\left( f^{\left( N\right) },\frac{1}{%
n^{\alpha }}\right) \frac{1}{n^{\alpha N}N!}
\end{equation*}%
(for $x\leq \frac{k}{n}$ or $x\geq \frac{k}{n}$).

Notice also for $x\leq \frac{k}{n}$ that 
\begin{equation*}
\left| \int_{x}^{\frac{k}{n}}\left( f^{\left( N\right) }\left( t\right)
-f^{\left( N\right) }\left( x\right) \right) \frac{\left( \frac{k}{n}%
-t\right) ^{N-1}}{\left( N-1\right) !}dt\right| \leq
\end{equation*}%
\begin{equation*}
\int_{x}^{\frac{k}{n}}\left| f^{\left( N\right) }\left( t\right) -f^{\left(
N\right) }\left( x\right) \right| \frac{\left( \frac{k}{n}-t\right) ^{N-1}}{%
\left( N-1\right) !}dt\leq
\end{equation*}%
\begin{equation*}
2\left\| f^{\left( N\right) }\right\| _{\infty }\int_{x}^{\frac{k}{n}}\frac{%
\left( \frac{k}{n}-t\right) ^{N-1}}{\left( N-1\right) !}dt=2\left\|
f^{\left( N\right) }\right\| _{\infty }\frac{\left( \frac{k}{n}-x\right) ^{N}%
}{N!}\leq 2\left\| f^{\left( N\right) }\right\| _{\infty }\frac{\left(
b-a\right) ^{N}}{N!}.
\end{equation*}%
Next assume $\frac{k}{n}\leq x$, then 
\begin{equation*}
\left| \int_{x}^{\frac{k}{n}}\left( f^{\left( N\right) }\left( t\right)
-f^{\left( N\right) }\left( x\right) \right) \frac{\left( \frac{k}{n}%
-t\right) ^{N-1}}{\left( N-1\right) !}dt\right| =
\end{equation*}%
\begin{equation*}
\left| \int_{\frac{k}{n}}^{x}\left( f^{\left( N\right) }\left( t\right)
-f^{\left( N\right) }\left( x\right) \right) \frac{\left( \frac{k}{n}%
-t\right) ^{N-1}}{\left( N-1\right) !}dt\right| \leq
\end{equation*}%
\begin{equation*}
\int_{\frac{k}{n}}^{x}\left| f^{\left( N\right) }\left( t\right) -f^{\left(
N\right) }\left( x\right) \right| \frac{\left( t-\frac{k}{n}\right) ^{N-1}}{%
\left( N-1\right) !}dt\leq
\end{equation*}%
\begin{equation*}
2\left\| f^{\left( N\right) }\right\| _{\infty }\int_{\frac{k}{n}}^{x}\frac{%
\left( t-\frac{k}{n}\right) ^{N-1}}{\left( N-1\right) !}dt=2\left\|
f^{\left( N\right) }\right\| _{\infty }\frac{\left( x-\frac{k}{n}\right) ^{N}%
}{N!}\leq 2\left\| f^{\left( N\right) }\right\| _{\infty }\frac{\left(
b-a\right) ^{N}}{N!}.
\end{equation*}%
Thus 
\begin{equation}
\left| \gamma \right| \leq 2\left\| f^{\left( N\right) }\right\| _{\infty }%
\frac{\left( b-a\right) ^{N}}{N!},  \tag{67}  \label{r67}
\end{equation}%
in all two cases.

Therefore 
\begin{equation*}
\Lambda _{n}\left( x\right) =\sum_{\substack{ k=\left\lceil na\right\rceil 
\\ \left| \frac{k}{n}-x\right| \leq \frac{1}{n^{\alpha }}}}^{\left\lfloor
nb\right\rfloor }\chi \left( nx-k\right) \gamma +\sum_{\substack{ %
k=\left\lceil na\right\rceil  \\ \left| \frac{k}{n}-x\right| >\frac{1}{%
n^{\alpha }}}}^{\left\lfloor nb\right\rfloor }\chi \left( nx-k\right) \gamma
.
\end{equation*}%
Hence 
\begin{equation*}
\left| \Lambda _{n}\left( x\right) \right| \leq \sum_{\substack{ %
k=\left\lceil na\right\rceil  \\ \left| \frac{k}{n}-x\right| \leq \frac{1}{%
n^{\alpha }}}}^{\left\lfloor nb\right\rfloor }\chi \left( nx-k\right) \left(
\omega _{1}\left( f^{\left( N\right) },\frac{1}{n^{\alpha }}\right) \frac{1}{%
N!n^{N\alpha }}\right) +
\end{equation*}%
\begin{equation*}
\left( \sum_{\substack{ k=\left\lceil na\right\rceil  \\ \left| \frac{k}{n}%
-x\right| >\frac{1}{n^{\alpha }}}}^{\left\lfloor nb\right\rfloor }\chi
\left( nx-k\right) \right) 2\left\| f^{\left( N\right) }\right\| _{\infty }%
\frac{\left( b-a\right) ^{N}}{N!}\leq
\end{equation*}%
\begin{equation*}
\omega _{1}\left( f^{\left( N\right) },\frac{1}{n^{\alpha }}\right) \frac{1}{%
N!n^{N\alpha }}+\left\| f^{\left( N\right) }\right\| _{\infty }\frac{\left(
b-a\right) ^{N}}{N!}\frac{1}{\sqrt{\pi }\left( n^{1-\alpha }-2\right)
e^{\left( n^{1-\alpha }-2\right) ^{2}}}.
\end{equation*}%
Consequently we have 
\begin{equation}
\left| \Lambda _{n}\left( x\right) \right| \leq \omega _{1}\left( f^{\left(
N\right) },\frac{1}{n^{\alpha }}\right) \frac{1}{n^{\alpha N}N!}+\frac{%
\left\| f^{\left( N\right) }\right\| _{\infty }\left( b-a\right) ^{N}}{N!%
\sqrt{\pi }\left( n^{1-\alpha }-2\right) e^{\left( n^{1-\alpha }-2\right)
^{2}}}.  \tag{68}  \label{r68}
\end{equation}%
We further see that 
\begin{equation*}
A_{n}^{\ast }\left( \left( \cdot -x\right) ^{j}\right) =\sum_{k=\left\lceil
na\right\rceil }^{\left\lfloor nb\right\rfloor }\chi \left( nx-k\right)
\left( \frac{k}{n}-x\right) ^{j}.
\end{equation*}%
Therefore 
\begin{equation*}
\left| A_{n}^{\ast }\left( \left( \cdot -x\right) ^{j}\right) \right| \leq
\sum_{k=\left\lceil na\right\rceil }^{\left\lfloor nb\right\rfloor }\chi
\left( nx-k\right) \left| \frac{k}{n}-x\right| ^{j}=
\end{equation*}%
\begin{equation*}
\sum_{\left\{ 
\begin{array}{c}
k=\left\lceil na\right\rceil \\ 
\left| \frac{k}{n}-x\right| \leq \frac{1}{n^{\alpha }}%
\end{array}%
\right. }^{\left\lfloor nb\right\rfloor }\chi \left( nx-k\right) \left| 
\frac{k}{n}-x\right| ^{j}+\sum_{\left\{ 
\begin{array}{c}
k=\left\lceil na\right\rceil \\ 
\left| \frac{k}{n}-x\right| >\frac{1}{n^{\alpha }}%
\end{array}%
\right. }^{\left\lfloor nb\right\rfloor }\chi \left( nx-k\right) \left| 
\frac{k}{n}-x\right| ^{j}\leq
\end{equation*}%
\begin{equation*}
\frac{1}{n^{\alpha j}}\sum_{\left\{ 
\begin{array}{c}
k=\left\lceil na\right\rceil \\ 
\left| \frac{k}{n}-x\right| \leq \frac{1}{n^{\alpha }}%
\end{array}%
\right. }^{\left\lfloor nb\right\rfloor }\chi \left( nx-k\right) +\left(
b-a\right) ^{j}\cdot \sum_{\left\{ 
\begin{array}{c}
k=\left\lceil na\right\rceil \\ 
\left| k-nx\right| >n^{1-\alpha }%
\end{array}%
\right. }^{\left\lfloor nb\right\rfloor }\chi \left( nx-k\right)
\end{equation*}%
\begin{equation*}
\leq \frac{1}{n^{\alpha j}}+\left( b-a\right) ^{j}\cdot \frac{1}{2\sqrt{\pi }%
\left( n^{1-\alpha }-2\right) e^{\left( n^{1-\alpha }-2\right) ^{2}}}.
\end{equation*}%
Hence 
\begin{equation}
\left| A_{n}^{\ast }\left( \left( \cdot -x\right) ^{j}\right) \right| \leq 
\frac{1}{n^{\alpha j}}+\frac{\left( b-a\right) ^{j}}{2\sqrt{\pi }\left(
n^{1-\alpha }-2\right) e^{\left( n^{1-\alpha }-2\right) ^{2}}},  \tag{69}
\label{r69}
\end{equation}%
for $j=1,...,N.$

Putting things together we have proved 
\begin{equation*}
\left| A_{n}^{\ast }\left( f,x\right) -f\left( x\right) \sum_{k=\left\lceil
na\right\rceil }^{\left\lfloor nb\right\rfloor }\chi \left( nx-k\right)
\right| \leq
\end{equation*}%
\begin{equation}
\dsum\limits_{j=1}^{N}\frac{\left| f^{\left( j\right) }\left( x\right)
\right| }{j!}\left[ \frac{1}{n^{\alpha j}}+\frac{\left( b-a\right) ^{j}}{2%
\sqrt{\pi }\left( n^{1-\alpha }-2\right) e^{\left( n^{1-\alpha }-2\right)
^{2}}}\right] +  \tag{70}  \label{r70}
\end{equation}%
\begin{equation*}
\left[ \omega _{1}\left( f^{\left( N\right) },\frac{1}{n^{\alpha }}\right) 
\frac{1}{n^{\alpha N}N!}+\frac{\left\| f^{\left( N\right) }\right\| _{\infty
}\left( b-a\right) ^{N}}{N!\sqrt{\pi }\left( n^{1-\alpha }-2\right)
e^{\left( n^{1-\alpha }-2\right) ^{2}}}\right] ,
\end{equation*}%
that is establishing theorem.
\end{proof}

\section{Fractional Neural Network Approximation}

We need

\begin{definition}
\label{d17}Let $\nu \geq 0$, $n=\left\lceil \nu \right\rceil $ ($\left\lceil
\cdot \right\rceil $ is the ceiling of the number), $f\in AC^{n}\left( \left[
a,b\right] \right) $ (space of functions $f$ with $f^{\left( n-1\right) }\in
AC\left( \left[ a,b\right] \right) $, absolutely continuous functions). We
call left Caputo fractional derivative (see \cite{16}, pp. 49-52, \cite{18}, 
\cite{23}) the function 
\begin{equation}
D_{\ast a}^{\nu }f\left( x\right) =\frac{1}{\Gamma \left( n-\nu \right) }%
\int_{a}^{x}\left( x-t\right) ^{n-\nu -1}f^{\left( n\right) }\left( t\right)
dt,  \tag{71}  \label{r71}
\end{equation}%
$\forall $ $x\in \left[ a,b\right] $, where $\Gamma $ is the gamma function $%
\Gamma \left( \nu \right) =\int_{0}^{\infty }e^{-t}t^{\nu -1}dt$, $\nu >0$.

Notice $D_{\ast a}^{\nu }f\in L_{1}\left( \left[ a,b\right] \right) $ and $%
D_{\ast a}^{\nu }f$ exists a.e.on $\left[ a,b\right] .$

We set $D_{\ast a}^{0}f\left( x\right) =f\left( x\right) $, $\forall $ $x\in %
\left[ a,b\right] .$
\end{definition}

\begin{lemma}
\label{l18}(\cite{6}) Let $\nu >0$, $\nu \notin \mathbb{N}$, $n=\left\lceil
\nu \right\rceil $, $f\in C^{n-1}\left( \left[ a,b\right] \right) $ and $%
f^{\left( n\right) }\in L_{\infty }\left( \left[ a,b\right] \right) $. Then $%
D_{\ast a}^{\nu }f\left( a\right) =0.$
\end{lemma}

\begin{definition}
\label{d19}(see also \cite{4}, \cite{17}, \cite{18}). Let $f\in AC^{m}\left( %
\left[ a,b\right] \right) $, $m=\left\lceil \alpha \right\rceil $, $\alpha
>0 $. The right Caputo fractional derivative of order $\alpha >0$ is given
by 
\begin{equation}
D_{b-}^{\alpha }f\left( x\right) =\frac{\left( -1\right) ^{m}}{\Gamma \left(
m-\alpha \right) }\int_{x}^{b}\left( \zeta -x\right) ^{m-\alpha -1}f^{\left(
m\right) }\left( \zeta \right) d\zeta ,  \tag{72}  \label{r72}
\end{equation}%
$\forall $ $x\in \left[ a,b\right] $. We set $D_{b-}^{0}f\left( x\right)
=f\left( x\right) .$ Notice $D_{b-}^{\alpha }f\in L_{1}\left( \left[ a,b%
\right] \right) $ and $D_{b-}^{\alpha }f$ exists a.e.on $\left[ a,b\right] .$
\end{definition}

\begin{lemma}
\label{l20}(\cite{6}) Let $f\in C^{m-1}\left( \left[ a,b\right] \right) $, $%
f^{\left( m\right) }\in L_{\infty }\left( \left[ a,b\right] \right) $, $%
m=\left\lceil \alpha \right\rceil $, $\alpha >0$. Then $D_{b-}^{\alpha
}f\left( b\right) =0.$
\end{lemma}

\begin{convention}
\label{c21}We assume that 
\begin{equation}
D_{\ast x_{0}}^{\alpha }f\left( x\right) =0\text{, \ for }x<x_{0},  \tag{73}
\label{r73}
\end{equation}%
and 
\begin{equation}
D_{x_{0}-}^{\alpha }f\left( x\right) =0\text{, for }x>x_{0},  \tag{74}
\label{r74}
\end{equation}%
for all $x,x_{0}\in (a,b].$
\end{convention}

We mention

\begin{proposition}
\label{p22}(\cite{6}) Let $f\in C^{n}\left( \left[ a,b\right] \right) $, $%
n=\left\lceil \nu \right\rceil $, $\nu >0$. Then $D_{\ast a}^{\nu }f\left(
x\right) $ is continuous in $x\in \left[ a,b\right] .$
\end{proposition}

Also we have

\begin{proposition}
\label{p23}(\cite{6}) Let $f\in C^{m}\left( \left[ a,b\right] \right) $, $%
m=\left\lceil \alpha \right\rceil $, $\alpha >0$. Then $D_{b-}^{\alpha
}f\left( x\right) $ is continuous in $x\in \left[ a,b\right] .$
\end{proposition}

We further mention

\begin{proposition}
\label{p24}(\cite{6}) Let $f\in C^{m-1}\left( \left[ a,b\right] \right) $, $%
f^{\left( m\right) }\in L_{\infty }\left( \left[ a,b\right] \right) $, $%
m=\left\lceil \alpha \right\rceil $, $\alpha >0$ and 
\begin{equation}
D_{\ast x_{0}}^{\alpha }f\left( x\right) =\frac{1}{\Gamma \left( m-\alpha
\right) }\int_{x_{0}}^{x}\left( x-t\right) ^{m-\alpha -1}f^{\left( m\right)
}\left( t\right) dt,  \tag{75}  \label{r75}
\end{equation}%
for all $x,x_{0}\in \left[ a,b\right] :x\geq x_{0}.$

Then $D_{\ast x_{0}}^{\alpha }f\left( x\right) $ is continuous in $x_{0}.$
\end{proposition}

\begin{proposition}
\label{p25}(\cite{6}) Let $f\in C^{m-1}\left( \left[ a,b\right] \right) $, $%
f^{\left( m\right) }\in L_{\infty }\left( \left[ a,b\right] \right) $, $%
m=\left\lceil \alpha \right\rceil $, $\alpha >0$ and 
\begin{equation}
D_{x_{0}-}^{\alpha }f\left( x\right) =\frac{\left( -1\right) ^{m}}{\Gamma
\left( m-\alpha \right) }\int_{x}^{x_{0}}\left( \zeta -x\right) ^{m-\alpha
-1}f^{\left( m\right) }\left( \zeta \right) d\zeta ,  \tag{76}  \label{r76}
\end{equation}%
for all $x,x_{0}\in \left[ a,b\right] :x\leq x_{0}.$

Then $D_{x_{0}-}^{\alpha }f\left( x\right) $ is continuous in $x_{0}.$
\end{proposition}

\begin{proposition}
\label{p26}(\cite{6}) Let $f\in C^{m}\left( \left[ a,b\right] \right) $, $%
m=\left\lceil \alpha \right\rceil $, $\alpha >0$, $x,x_{0}\in \left[ a,b%
\right] $. Then $D_{\ast x_{0}}^{\alpha }f\left( x\right) $, $%
D_{x_{0}-}^{\alpha }f\left( x\right) $ are jointly continuous functions in $%
\left( x,x_{0}\right) $ from $\left[ a,b\right] ^{2}$ into $\mathbb{R}$.
\end{proposition}

We recall

\begin{theorem}
\label{t27}(\cite{6}) Let $f:\left[ a,b\right] ^{2}\rightarrow \mathbb{R}$
be jointly continuous. Consider 
\begin{equation}
G\left( x\right) =\omega _{1}\left( f\left( \cdot ,x\right) ,\delta ,\left[
x,b\right] \right) ,  \tag{77}  \label{r77}
\end{equation}%
$\delta >0$, $x\in \left[ a,b\right] .$

Then $G$ is continuous in $x\in \left[ a,b\right] .$
\end{theorem}

Also it holds

\begin{theorem}
\label{t28}(\cite{6}) Let $f:\left[ a,b\right] ^{2}\rightarrow \mathbb{R}$
be jointly continuous. Then 
\begin{equation}
H\left( x\right) =\omega _{1}\left( f\left( \cdot ,x\right) ,\delta ,\left[
a,x\right] \right) ,  \tag{78}  \label{r78}
\end{equation}%
$x\in \left[ a,b\right] $, is continuous in $x\in \left[ a,b\right] $, $%
\delta >0.$
\end{theorem}

We need

\begin{remark}
\label{re29}(\cite{6}) Let $f\in C^{n-1}\left( \left[ a,b\right] \right) $, $%
f^{\left( n\right) }\in L_{\infty }\left( \left[ a,b\right] \right) $, $%
n=\left\lceil \nu \right\rceil $, $\nu >0$, $\nu \notin \mathbb{N}$. Then we
have 
\begin{equation}
\left| D_{\ast a}^{\nu }f\left( x\right) \right| \leq \frac{\left\|
f^{\left( n\right) }\right\| _{\infty }}{\Gamma \left( n-\nu +1\right) }%
\left( x-a\right) ^{n-\nu }\text{, \ }\forall \text{ }x\in \left[ a,b\right]
.  \tag{79}  \label{r79}
\end{equation}%
Thus we observe 
\begin{equation*}
\omega _{1}\left( D_{\ast a}^{\nu }f,\delta \right) =\underset{\left|
x-y\right| \leq \delta }{\underset{x,y\in \left[ a,b\right] }{\sup }}\left|
D_{\ast a}^{\nu }f\left( x\right) -D_{\ast a}^{\nu }f\left( y\right) \right|
\end{equation*}%
\begin{equation*}
\leq \underset{\left| x-y\right| \leq \delta }{\underset{x,y\in \left[ a,b%
\right] }{\sup }}\left( \frac{\left\| f^{\left( n\right) }\right\| _{\infty }%
}{\Gamma \left( n-\nu +1\right) }\left( x-a\right) ^{n-\nu }+\frac{\left\|
f^{\left( n\right) }\right\| _{\infty }}{^{\Gamma \left( n-\nu +1\right) }}%
\left( y-a\right) ^{n-\nu }\right)
\end{equation*}%
\begin{equation*}
\leq \frac{2\left\| f^{\left( n\right) }\right\| _{\infty }}{\Gamma \left(
n-\nu +1\right) }\left( b-a\right) ^{n-\nu }.
\end{equation*}%
Consequently 
\begin{equation}
\omega _{1}\left( D_{\ast a}^{\nu }f,\delta \right) \leq \frac{2\left\|
f^{\left( n\right) }\right\| _{\infty }}{\Gamma \left( n-\nu +1\right) }%
\left( b-a\right) ^{n-\nu }.  \tag{80}  \label{r80}
\end{equation}

Similarly, let $f\in C^{m-1}\left( \left[ a,b\right] \right) $, $f^{\left(
m\right) }\in L_{\infty }\left( \left[ a,b\right] \right) $, $m=\left\lceil
\alpha \right\rceil $, $\alpha >0$, $\alpha \notin \mathbb{N}$, then 
\begin{equation}
\omega _{1}\left( D_{b-}^{\alpha }f,\delta \right) \leq \frac{2\left\|
f^{\left( m\right) }\right\| _{\infty }}{\Gamma \left( m-\alpha +1\right) }%
\left( b-a\right) ^{m-\alpha }.  \tag{81}  \label{r81}
\end{equation}%
So for $f\in C^{m-1}\left( \left[ a,b\right] \right) $, $f^{\left( m\right)
}\in L_{\infty }\left( \left[ a,b\right] \right) $, $m=\left\lceil \alpha
\right\rceil $, $\alpha >0$, $\alpha \notin \mathbb{N}$, we find 
\begin{equation}
\underset{x_{0}\in \left[ a,b\right] }{\sup }\omega _{1}\left( D_{\ast
x_{0}}^{\alpha }f,\delta \right) _{\left[ x_{0},b\right] }\leq \frac{%
2\left\| f^{\left( m\right) }\right\| _{\infty }}{\Gamma \left( m-\alpha
+1\right) }\left( b-a\right) ^{m-\alpha },  \tag{82}  \label{r82}
\end{equation}%
and 
\begin{equation}
\underset{x_{0}\in \left[ a,b\right] }{\sup }\omega _{1}\left(
D_{x_{0}-}^{\alpha }f,\delta \right) _{\left[ a,x_{0}\right] }\leq \frac{%
2\left\| f^{\left( m\right) }\right\| _{\infty }}{\Gamma \left( m-\alpha
+1\right) }\left( b-a\right) ^{m-\alpha }.  \tag{83}  \label{r83}
\end{equation}

By Proposition 15.114, p. 388 of \cite{5}, we get here that $D_{\ast
x_{0}}^{\alpha }f\in C\left( \left[ x_{0},b\right] \right) $, and by \cite{8}
we obtain that $D_{x_{0}-}^{\alpha }f\in C\left( \left[ a,x_{0}\right]
\right) $.
\end{remark}

Here comes our main fractional result

\begin{theorem}
\label{t30}Let $\alpha >0$, $N=\left\lceil \alpha \right\rceil $, $\alpha
\notin \mathbb{N},$ $f\in AC^{N}\left( \left[ a,b\right] \right) $, with $%
f^{\left( N\right) }\in L_{\infty }\left( \left[ a,b\right] \right) ,$ $%
0<\beta <1$, $x\in \left[ a,b\right] $, $n\in \mathbb{N}$, $n^{1-\beta }\geq
3$. Then

i) 
\begin{equation}
\left| A_{n}\left( f,x\right) -\sum_{j=1}^{N-1}\frac{f^{\left( j\right)
}\left( x\right) }{j!}A_{n}\left( \left( \cdot -x\right) ^{j}\right) \left(
x\right) -f\left( x\right) \right| \leq  \tag{84}  \label{r84}
\end{equation}%
\begin{equation*}
\frac{\left( 4.019\right) }{\Gamma \left( \alpha +1\right) }\cdot \left\{ 
\frac{\left( \omega _{1}\left( D_{x-}^{\alpha }f,\frac{1}{n^{\beta }}\right)
_{\left[ a,x\right] }+\omega _{1}\left( D_{\ast x}^{\alpha }f,\frac{1}{%
n_{\beta }}\right) _{\left[ x,b\right] }\right) }{n^{\alpha \beta }}+\right.
\end{equation*}%
\begin{equation*}
\frac{1}{2\sqrt{\pi }\left( n^{1-\beta }-2\right) e^{\left( n^{1-\beta
}-2\right) ^{2}}}\cdot
\end{equation*}%
\begin{equation*}
\left. \left( \left\| D_{x-}^{\alpha }f\right\| _{\infty ,\left[ a,x\right]
}\left( x-a\right) ^{\alpha }+\left\| D_{\ast x}^{\alpha }f\right\| _{\infty
,\left[ x,b\right] }\left( b-x\right) ^{\alpha }\right) \right\} ,
\end{equation*}

ii) if $f^{\left( j\right) }\left( x\right) =0$, for $j=1,...,N-1$, we have 
\begin{equation}
\left| A_{n}\left( f,x\right) -f\left( x\right) \right| \leq \frac{\left(
4.019\right) }{\Gamma \left( \alpha +1\right) }\cdot  \tag{85}  \label{r85}
\end{equation}%
\begin{equation*}
\left\{ \frac{\left( \omega _{1}\left( D_{x-}^{\alpha }f,\frac{1}{n^{\beta }}%
\right) _{\left[ a,x\right] }+\omega _{1}\left( D_{\ast x}^{\alpha }f,\frac{1%
}{n_{\beta }}\right) _{\left[ x,b\right] }\right) }{n^{\alpha \beta }}%
+\right.
\end{equation*}%
\begin{equation*}
\frac{1}{2\sqrt{\pi }\left( n^{1-\beta }-2\right) e^{\left( n^{1-\beta
}-2\right) ^{2}}}\cdot
\end{equation*}%
\begin{equation*}
\left. \left( \left\| D_{x-}^{\alpha }f\right\| _{\infty ,\left[ a,x\right]
}\left( x-a\right) ^{\alpha }+\left\| D_{\ast x}^{\alpha }f\right\| _{\infty
,\left[ x,b\right] }\left( b-x\right) ^{\alpha }\right) \right\} ,
\end{equation*}%
when $\alpha >1$ notice here the extremely high rate of convergence at $%
n^{-\left( \alpha +1\right) \beta },$

iii) 
\begin{equation}
\left| A_{n}\left( f,x\right) -f\left( x\right) \right| \leq \left(
4.019\right) \cdot  \tag{86}  \label{r86}
\end{equation}%
\begin{equation*}
\left\{ \sum_{j=1}^{N-1}\frac{\left| f^{\left( j\right) }\left( x\right)
\right| }{j!}\left\{ \frac{1}{n^{\beta j}}+\left( b-a\right) ^{j}\frac{1}{2%
\sqrt{\pi }\left( n^{1-\beta }-2\right) e^{\left( n^{1-\beta }-2\right) ^{2}}%
}\right\} +\right.
\end{equation*}%
\begin{equation*}
\frac{1}{\Gamma \left( \alpha +1\right) }\left\{ \frac{\left( \omega
_{1}\left( D_{x-}^{\alpha }f,\frac{1}{n^{\beta }}\right) _{\left[ a,x\right]
}+\omega _{1}\left( D_{\ast x}^{\alpha }f,\frac{1}{n_{\beta }}\right) _{%
\left[ x,b\right] }\right) }{n^{\alpha \beta }}\right. +
\end{equation*}%
\begin{equation*}
\frac{1}{2\sqrt{\pi }\left( n^{1-\beta }-2\right) e^{\left( n^{1-\beta
}-2\right) ^{2}}}\cdot
\end{equation*}%
\begin{equation*}
\left. \left. \left( \left\| D_{x-}^{\alpha }f\right\| _{\infty ,\left[ a,x%
\right] }\left( x-a\right) ^{\alpha }+\left\| D_{\ast x}^{\alpha }f\right\|
_{\infty ,\left[ x,b\right] }\left( b-x\right) ^{\alpha }\right) \right\}
\right\} ,
\end{equation*}%
$\forall $ $x\in \left[ a,b\right] ,$

and

iv)%
\begin{equation*}
\left\| A_{n}f-f\right\| _{\infty }\leq \left( 4.019\right) \cdot
\end{equation*}%
\begin{equation*}
\left\{ \sum_{j=1}^{N-1}\frac{\left\| f^{\left( j\right) }\right\| _{\infty }%
}{j!}\left\{ \frac{1}{n^{\beta j}}+\left( b-a\right) ^{j}\frac{1}{2\sqrt{\pi 
}\left( n^{1-\beta }-2\right) e^{\left( n^{1-\beta }-2\right) ^{2}}}\right\}
+\right.
\end{equation*}%
\begin{equation}
\frac{1}{\Gamma \left( \alpha +1\right) }\left\{ \frac{\left( \underset{x\in %
\left[ a,b\right] }{\sup }\omega _{1}\left( D_{x-}^{\alpha }f,\frac{1}{%
n^{\beta }}\right) _{\left[ a,x\right] }+\underset{x\in \left[ a,b\right] }{%
\sup }\omega _{1}\left( D_{\ast x}^{\alpha }f,\frac{1}{n_{\beta }}\right) _{%
\left[ x,b\right] }\right) }{n^{\alpha \beta }}\right. +  \tag{87}
\label{r87}
\end{equation}%
\begin{equation*}
\frac{1}{2\sqrt{\pi }\left( n^{1-\beta }-2\right) e^{\left( n^{1-\beta
}-2\right) ^{2}}}\cdot
\end{equation*}%
\begin{equation*}
\left. \left. \left( b-a\right) ^{\alpha }\left( \underset{x\in \left[ a,b%
\right] }{\sup }\left\| D_{x-}^{\alpha }f\right\| _{\infty ,\left[ a,x\right]
}+\underset{x\in \left[ a,b\right] }{\sup }\left\| D_{\ast x}^{\alpha
}f\right\| _{\infty ,\left[ x,b\right] }\right) \right\} \right\} .
\end{equation*}%
Above, when $N=1$ the sum $\sum_{j=1}^{N-1}\cdot =0$.

As we see here we obtain fractionally type pointwise and uniform convergence
with rates of $A_{n}\rightarrow I$ the unit operator, as $n\rightarrow
\infty .$
\end{theorem}

\begin{proof}
Let $x\in \left[ a,b\right] $. We have that $D_{x-}^{\alpha }f\left(
x\right) =D_{\ast x}^{\alpha }f\left( x\right) =0$.

From \cite{16}, p. 54, we get by the left Caputo fractional Taylor formula
that 
\begin{equation}
f\left( \frac{k}{n}\right) =\sum_{j=0}^{N-1}\frac{f^{\left( j\right) }\left(
x\right) }{j!}\left( \frac{k}{n}-x\right) ^{j}+  \tag{88}  \label{r88}
\end{equation}%
\begin{equation*}
\frac{1}{\Gamma \left( \alpha \right) }\int_{x}^{\frac{k}{n}}\left( \frac{k}{%
n}-J\right) ^{\alpha -1}\left( D_{\ast x}^{\alpha }f\left( J\right) -D_{\ast
x}^{\alpha }f\left( x\right) \right) dJ,
\end{equation*}%
for all $x\leq \frac{k}{n}\leq b.$

Also from \cite{4}, using the right Caputo fractional Taylor formula we get 
\begin{equation}
f\left( \frac{k}{n}\right) =\sum_{j=0}^{N-1}\frac{f^{\left( j\right) }\left(
x\right) }{j!}\left( \frac{k}{n}-x\right) ^{j}+  \tag{89}  \label{r89}
\end{equation}%
\begin{equation*}
\frac{1}{\Gamma \left( \alpha \right) }\int_{\frac{k}{n}}^{x}\left( J-\frac{k%
}{n}\right) ^{\alpha -1}\left( D_{x-}^{\alpha }f\left( J\right)
-D_{x-}^{\alpha }f\left( x\right) \right) dJ,
\end{equation*}%
for all $a\leq \frac{k}{n}\leq x.$

Hence we have 
\begin{equation}
f\left( \frac{k}{n}\right) \chi \left( nx-k\right) =\sum_{j=0}^{N-1}\frac{%
f^{\left( j\right) }\left( x\right) }{j!}\chi \left( nx-k\right) \left( 
\frac{k}{n}-x\right) ^{j}+  \tag{90}  \label{r90}
\end{equation}%
\begin{equation*}
\frac{\chi \left( nx-k\right) }{\Gamma \left( \alpha \right) }\int_{x}^{%
\frac{k}{n}}\left( \frac{k}{n}-J\right) ^{\alpha -1}\left( D_{\ast
x}^{\alpha }f\left( J\right) -D_{\ast x}^{\alpha }f\left( x\right) \right)
dJ,
\end{equation*}%
for all $x\leq \frac{k}{n}\leq b$, iff $\left\lceil nx\right\rceil \leq
k\leq \left\lfloor nb\right\rfloor $, and 
\begin{equation}
f\left( \frac{k}{n}\right) \chi \left( nx-k\right) =\sum_{j=0}^{N-1}\frac{%
f^{\left( j\right) }\left( x\right) }{j!}\chi \left( nx-k\right) \left( 
\frac{k}{n}-x\right) ^{j}+  \tag{91}  \label{r91}
\end{equation}%
\begin{equation*}
\frac{\chi \left( nx-k\right) }{\Gamma \left( \alpha \right) }\int_{\frac{k}{%
n}}^{x}\left( J-\frac{k}{n}\right) ^{\alpha -1}\left( D_{x-}^{\alpha
}f\left( J\right) -D_{x-}^{\alpha }f\left( x\right) \right) dJ,
\end{equation*}%
for all $a\leq \frac{k}{n}\leq x$, iff $\left\lceil na\right\rceil \leq
k\leq \left\lfloor nx\right\rfloor $.

We have that $\left\lceil nx\right\rceil \leq \left\lfloor nx\right\rfloor
+1.$

Therefore it holds 
\begin{equation}
\sum_{k=\left\lfloor nx\right\rfloor +1}^{\left\lfloor nb\right\rfloor
}f\left( \frac{k}{n}\right) \chi \left( nx-k\right) =\sum_{j=0}^{N-1}\frac{%
f^{\left( j\right) }\left( x\right) }{j!}\sum_{k=\left\lfloor
nx\right\rfloor +1}^{\left\lfloor nb\right\rfloor }\chi \left( nx-k\right)
\left( \frac{k}{n}-x\right) ^{j}+  \tag{92}  \label{r92}
\end{equation}%
\begin{equation*}
\frac{1}{\Gamma \left( \alpha \right) }\sum_{k=\left\lfloor nx\right\rfloor
+1}^{\left\lfloor nb\right\rfloor }\chi \left( nx-k\right) \int_{x}^{\frac{k%
}{n}}\left( \frac{k}{n}-J\right) ^{\alpha -1}\left( D_{\ast x}^{\alpha
}f\left( J\right) -D_{\ast x}^{\alpha }f\left( x\right) \right) dJ,
\end{equation*}%
and 
\begin{equation}
\sum_{k=\left\lceil na\right\rceil }^{\left\lfloor nx\right\rfloor }f\left( 
\frac{k}{n}\right) \chi \left( nx-k\right) =\sum_{j=0}^{N-1}\frac{f^{\left(
j\right) }\left( x\right) }{j!}\sum_{k=\left\lceil na\right\rceil
}^{\left\lfloor nx\right\rfloor }\chi \left( nx-k\right) \left( \frac{k}{n}%
-x\right) ^{j}+  \tag{93}  \label{r93}
\end{equation}%
\begin{equation*}
\frac{1}{\Gamma \left( \alpha \right) }\sum_{k=\left\lceil na\right\rceil
}^{\left\lfloor nx\right\rfloor }\chi \left( nx-k\right) \int_{\frac{k}{n}%
}^{x}\left( J-\frac{k}{n}\right) ^{\alpha -1}\left( D_{x-}^{\alpha }f\left(
J\right) -D_{x-}^{\alpha }f\left( x\right) \right) dJ.
\end{equation*}%
Adding the last two equalities (\ref{r92}) and (\ref{r93}) we obtain 
\begin{equation}
A_{n}^{\ast }\left( f,x\right) =\sum_{k=\left\lceil na\right\rceil
}^{\left\lfloor nb\right\rfloor }f\left( \frac{k}{n}\right) \chi \left(
nx-k\right) =  \tag{94}  \label{r94}
\end{equation}%
\begin{equation*}
\sum_{j=0}^{N-1}\frac{f^{\left( j\right) }\left( x\right) }{j!}%
\sum_{k=\left\lceil na\right\rceil }^{\left\lfloor nb\right\rfloor }\chi
\left( nx-k\right) \left( \frac{k}{n}-x\right) ^{j}+
\end{equation*}%
\begin{equation*}
\frac{1}{\Gamma \left( \alpha \right) }\left\{ \sum_{k=\left\lceil
na\right\rceil }^{\left\lfloor nx\right\rfloor }\chi \left( nx-k\right)
\int_{\frac{k}{n}}^{x}\left( J-\frac{k}{n}\right) ^{\alpha -1}\left(
D_{x-}^{\alpha }f\left( J\right) -D_{x-}^{\alpha }f\left( x\right) \right)
dJ+\right.
\end{equation*}%
\begin{equation*}
\left. \sum_{k=\left\lfloor nx\right\rfloor +1}^{\left\lfloor
nb\right\rfloor }\chi \left( nx-k\right) \int_{x}^{\frac{k}{n}}\left( \frac{k%
}{n}-J\right) ^{\alpha -1}\left( D_{\ast x}^{\alpha }f\left( J\right)
-D_{\ast x}^{\alpha }f\left( x\right) \right) dJ\right\} .
\end{equation*}%
So we have derived 
\begin{equation}
A_{n}^{\ast }\left( f,x\right) -f\left( x\right) \left( \sum_{k=\left\lceil
na\right\rceil }^{\left\lfloor nb\right\rfloor }\chi \left( nx-k\right)
\right) =  \tag{95}  \label{r95}
\end{equation}%
\begin{equation*}
\sum_{j=1}^{N-1}\frac{f^{\left( j\right) }\left( x\right) }{j!}A_{n}^{\ast
}\left( \left( \cdot -x\right) ^{j}\right) \left( x\right) +\theta
_{n}\left( x\right) ,
\end{equation*}%
where 
\begin{equation*}
\theta _{n}\left( x\right) :=\frac{1}{\Gamma \left( \alpha \right) }\left\{
\sum_{k=\left\lceil na\right\rceil }^{\left\lfloor nx\right\rfloor }\chi
\left( nx-k\right) \int_{\frac{k}{n}}^{x}\left( J-\frac{k}{n}\right)
^{\alpha -1}\left( D_{x-}^{\alpha }f\left( J\right) -D_{x-}^{\alpha }f\left(
x\right) \right) dJ\right.
\end{equation*}%
\begin{equation}
\left. +\sum_{k=\left\lfloor nx\right\rfloor +1}^{\left\lfloor
nb\right\rfloor }\chi \left( nx-k\right) \int_{x}^{\frac{k}{n}}\left( \frac{k%
}{n}-J\right) ^{\alpha -1}\left( D_{\ast x}^{\alpha }f\left( J\right)
-D_{\ast x}^{\alpha }f\left( x\right) \right) dJ\right\} .  \tag{96}
\label{r96}
\end{equation}%
We set 
\begin{equation}
\theta _{1n}\left( x\right) :=\frac{1}{\Gamma \left( \alpha \right) }%
\sum_{k=\left\lceil na\right\rceil }^{\left\lfloor nx\right\rfloor }\chi
\left( nx-k\right) \int_{\frac{k}{n}}^{x}\left( J-\frac{k}{n}\right)
^{\alpha -1}\left( D_{x-}^{\alpha }f\left( J\right) -D_{x-}^{\alpha }f\left(
x\right) \right) dJ,  \tag{97}  \label{r97}
\end{equation}%
and 
\begin{equation}
\theta _{2n}:=\frac{1}{\Gamma \left( \alpha \right) }\sum_{k=\left\lfloor
nx\right\rfloor +1}^{\left\lfloor nb\right\rfloor }\chi \left( nx-k\right)
\int_{x}^{\frac{k}{n}}\left( \frac{k}{n}-J\right) ^{\alpha -1}\left( D_{\ast
x}^{\alpha }f\left( J\right) -D_{\ast x}^{\alpha }f\left( x\right) \right)
dJ,  \tag{98}  \label{r98}
\end{equation}%
i.e. 
\begin{equation}
\theta _{n}\left( x\right) =\theta _{1n}\left( x\right) +\theta _{2n}\left(
x\right) .  \tag{99}  \label{r99}
\end{equation}%
We assume $b-a>\frac{1}{n^{\beta }}$, $0<\beta <1$, which is always the case
for large enough $n\in \mathbb{N}$, that is when $n>\left\lceil \left(
b-a\right) ^{-\frac{1}{\beta }}\right\rceil .$ It is always true that either 
$\left| \frac{k}{n}-x\right| \leq \frac{1}{n^{\beta }}$ or $\left| \frac{k}{n%
}-x\right| >\frac{1}{n^{\beta }}$.

For $k=\left\lceil na\right\rceil ,...,\left\lfloor nx\right\rfloor $, we
consider 
\begin{equation}
\gamma _{1k}:=\left| \int_{\frac{k}{n}}^{x}\left( J-\frac{k}{n}\right)
^{\alpha -1}\left( D_{x-}^{\alpha }f\left( J\right) -D_{x-}^{\alpha }f\left(
x\right) \right) dJ\right|  \tag{100}  \label{r100}
\end{equation}%
\begin{equation*}
=\left| \int_{\frac{k}{n}}^{x}\left( J-\frac{k}{n}\right) ^{\alpha
-1}D_{x-}^{\alpha }f\left( J\right) dJ\right| \leq \int_{\frac{k}{n}%
}^{x}\left( J-\frac{k}{n}\right) ^{\alpha -1}\left| D_{x-}^{\alpha }f\left(
J\right) \right| dJ
\end{equation*}%
\begin{equation}
\leq \left\| D_{x-}^{\alpha }f\right\| _{\infty ,\left[ a,x\right] }\frac{%
\left( x-\frac{\kappa }{n}\right) ^{\alpha }}{\alpha }\leq \left\|
D_{x-}^{\alpha }f\right\| _{\infty ,\left[ a,x\right] }\frac{\left(
x-a\right) ^{\alpha }}{\alpha }.  \tag{101}  \label{r101}
\end{equation}%
That is 
\begin{equation}
\gamma _{1k}\leq \left\| D_{x-}^{\alpha }f\right\| _{\infty ,\left[ a,x%
\right] }\frac{\left( x-a\right) ^{\alpha }}{\alpha },  \tag{102}
\label{r102}
\end{equation}%
for $k=\left\lceil na\right\rceil ,...,\left\lfloor nx\right\rfloor .$

Also we have in case of $\left| \frac{k}{n}-x\right| \leq \frac{1}{n^{\beta }%
}$ that 
\begin{equation}
\gamma _{1k}\leq \int_{\frac{k}{n}}^{x}\left( J-\frac{k}{n}\right) ^{\alpha
-1}\left| D_{x-}^{\alpha }f\left( J\right) -D_{x-}^{\alpha }f\left( x\right)
\right| dJ  \tag{103}  \label{r103}
\end{equation}%
\begin{equation*}
\leq \int_{\frac{k}{n}}^{x}\left( J-\frac{k}{n}\right) ^{\alpha -1}\omega
_{1}\left( D_{x-}^{\alpha }f,\left| J-x\right| \right) _{\left[ a,x\right]
}dJ
\end{equation*}%
\begin{equation*}
\leq \omega _{1}\left( D_{x-}^{\alpha }f,\left| x-\frac{k}{n}\right| \right)
_{\left[ a,x\right] }\int_{\frac{k}{n}}^{x}\left( J-\frac{k}{n}\right)
^{\alpha -1}dJ
\end{equation*}%
\begin{equation}
\leq \omega _{1}\left( D_{x-}^{\alpha }f,\frac{1}{n^{\beta }}\right) _{\left[
a,x\right] }\frac{\left( x-\frac{k}{n}\right) ^{\alpha }}{\alpha }\leq
\omega _{1}\left( D_{x-}^{\alpha }f,\frac{1}{n^{\beta }}\right) _{\left[ a,x%
\right] }\frac{1}{\alpha n^{a\beta }}.  \tag{104}  \label{r104}
\end{equation}%
That is when $\left| \frac{k}{n}-x\right| \leq \frac{1}{n^{\beta }}$, then 
\begin{equation}
\gamma _{1k}\leq \frac{\omega _{1}\left( D_{x-}^{\alpha }f,\frac{1}{n^{\beta
}}\right) _{\left[ a,x\right] }}{\alpha n^{a\beta }}.  \tag{105}
\label{r105}
\end{equation}%
Consequently we obtain

\begin{equation*}
\left| \theta _{1n}\left( x\right) \right| \leq \frac{1}{\Gamma \left(
\alpha \right) }\sum_{k=\left\lceil na\right\rceil }^{\left\lfloor
nx\right\rfloor }\chi \left( nx-k\right) \gamma _{1k}=
\end{equation*}%
\begin{equation*}
\frac{1}{\Gamma \left( \alpha \right) }\left\{ \sum_{\left\{ 
\begin{array}{c}
k=\left\lceil na\right\rceil \\ 
:\left| \frac{k}{n}-x\right| \leq \frac{1}{n^{\beta }}%
\end{array}%
\right. }^{\left\lfloor nx\right\rfloor }\chi \left( nx-k\right) \gamma
_{1k}+\sum_{\left\{ 
\begin{array}{c}
k=\left\lceil na\right\rceil \\ 
:\left| \frac{k}{n}-x\right| >\frac{1}{n^{\beta }}%
\end{array}%
\right. }^{\left\lfloor nx\right\rfloor }\chi \left( nx-k\right) \gamma
_{1k}\right\} \leq
\end{equation*}%
\begin{equation*}
\frac{1}{\Gamma \left( \alpha \right) }\left\{ \left( \sum_{\left\{ 
\begin{array}{c}
k=\left\lceil na\right\rceil \\ 
:\left| \frac{k}{n}-x\right| \leq \frac{1}{n^{\beta }}%
\end{array}%
\right. }^{\left\lfloor nx\right\rfloor }\chi \left( nx-k\right) \right) 
\frac{\omega _{1}\left( D_{x-}^{\alpha }f,\frac{1}{n^{\beta }}\right) _{%
\left[ a,x\right] }}{\alpha n^{\alpha \beta }}+\right.
\end{equation*}%
\begin{equation}
\left. \left( \sum_{\left\{ 
\begin{array}{c}
k=\left\lceil na\right\rceil \\ 
:\left| \frac{k}{n}-x\right| >\frac{1}{n^{\beta }}%
\end{array}%
\right. }^{\left\lfloor nx\right\rfloor }\chi \left( nx-k\right) \right)
\left\| D_{x-}^{\alpha }f\right\| _{\infty ,\left[ a,x\right] }\frac{\left(
x-a\right) ^{\alpha }}{\alpha }\right\} \leq  \tag{106}  \label{r106}
\end{equation}%
\begin{equation*}
\frac{1}{\Gamma \left( \alpha +1\right) }\left\{ \frac{\omega _{1}\left(
D_{x-}^{\alpha }f,\frac{1}{n^{\beta }}\right) _{\left[ a,x\right] }}{%
n^{\alpha \beta }}+\right.
\end{equation*}%
\begin{equation}
\left. \left( \sum_{\left\{ 
\begin{array}{c}
k=-\infty \\ 
:\left| nx-k\right| >n^{1-\beta }%
\end{array}%
\right. }^{\infty }\chi \left( nx-k\right) \right) \left\| D_{x-}^{\alpha
}f\right\| _{\infty ,\left[ a,x\right] }\left( x-a\right) ^{\alpha }\right\} 
\overset{\text{(\ref{r18})}}{\leq }  \tag{107}  \label{r107}
\end{equation}%
\begin{equation*}
\frac{1}{\Gamma \left( \alpha +1\right) }\left\{ \frac{\omega _{1}\left(
D_{x-}^{\alpha }f,\frac{1}{n^{\beta }}\right) _{\left[ a,x\right] }}{%
n^{\alpha \beta }}+\right.
\end{equation*}%
\begin{equation*}
\left. \frac{1}{2\sqrt{\pi }\left( n^{1-\beta }-2\right) e^{\left(
n^{1-\beta }-2\right) ^{2}}}\left\| D_{x-}^{\alpha }f\right\| _{\infty ,%
\left[ a,x\right] }\left( x-a\right) ^{\alpha }\right\} .
\end{equation*}%
So we have proved that 
\begin{equation}
\left| \theta _{1n}\left( x\right) \right| \leq \frac{1}{\Gamma \left(
\alpha +1\right) }\left\{ \frac{\omega _{1}\left( D_{x-}^{\alpha }f,\frac{1}{%
n^{\beta }}\right) _{\left[ a,x\right] }}{n^{\alpha \beta }}+\right. 
\tag{108}  \label{r108}
\end{equation}%
\begin{equation*}
\left. \frac{1}{2\sqrt{\pi }\left( n^{1-\beta }-2\right) e^{\left(
n^{1-\beta }-2\right) ^{2}}}\left\| D_{x-}^{\alpha }f\right\| _{\infty ,%
\left[ a,x\right] }\left( x-a\right) ^{\alpha }\right\} .
\end{equation*}%
Next when $k=\left\lfloor nx\right\rfloor +1,...,\left\lfloor
nb\right\rfloor $ we consider 
\begin{equation*}
\gamma _{2k}:=\left| \int_{x}^{\frac{k}{n}}\left( \frac{k}{n}-J\right)
^{\alpha -1}\left( D_{\ast x}^{\alpha }f\left( J\right) -D_{\ast x}^{\alpha
}f\left( x\right) \right) dJ\right| \leq
\end{equation*}%
\begin{equation*}
\int_{x}^{\frac{k}{n}}\left( \frac{k}{n}-J\right) ^{\alpha -1}\left| D_{\ast
x}^{\alpha }f\left( J\right) -D_{\ast x}^{\alpha }f\left( x\right) \right|
dJ=
\end{equation*}%
\begin{equation}
\int_{x}^{\frac{k}{n}}\left( \frac{k}{n}-J\right) ^{\alpha -1}\left| D_{\ast
x}^{\alpha }f\left( J\right) \right| dJ\leq \left\| D_{\ast x}^{\alpha
}f\right\| _{\infty ,\left[ x,b\right] }\frac{\left( \frac{k}{n}-x\right)
^{\alpha }}{\alpha }\leq  \tag{109}  \label{r109}
\end{equation}%
\begin{equation}
\left\| D_{\ast x}^{\alpha }f\right\| _{\infty ,\left[ x,b\right] }\frac{%
\left( b-x\right) ^{\alpha }}{\alpha }.  \tag{110}  \label{r110}
\end{equation}%
Therefore when $k=\left\lfloor nx\right\rfloor +1,...,\left\lfloor
nb\right\rfloor $ we get that 
\begin{equation}
\gamma _{2k}\leq \left\| D_{\ast x}^{\alpha }f\right\| _{\infty ,\left[ x,b%
\right] }\frac{\left( b-x\right) ^{\alpha }}{\alpha }.  \tag{111}
\label{r111}
\end{equation}%
In case of $\left| \frac{k}{n}-x\right| \leq \frac{1}{n^{\beta }}$, we get 
\begin{equation}
\gamma _{2k}\leq \int_{x}^{\frac{k}{n}}\left( \frac{k}{n}-J\right) ^{\alpha
-1}\omega _{1}\left( D_{\ast x}^{\alpha }f,\left| J-x\right| \right) _{\left[
x,b\right] }dJ\leq  \tag{112}  \label{r112}
\end{equation}%
\begin{equation*}
\omega _{1}\left( D_{\ast x}^{\alpha }f,\left| \frac{k}{n}-x\right| \right)
_{\left[ x,b\right] }\int_{x}^{\frac{k}{n}}\left( \frac{k}{n}-J\right)
^{\alpha -1}dJ\leq
\end{equation*}%
\begin{equation}
\omega _{1}\left( D_{\ast x}^{\alpha }f,\frac{1}{n^{\beta }}\right) _{\left[
x,b\right] }\frac{\left( \frac{k}{n}-x\right) ^{\alpha }}{\alpha }\leq
\omega _{1}\left( D_{\ast x}^{\alpha }f,\frac{1}{n^{\beta }}\right) _{\left[
x,b\right] }\frac{1}{\alpha n^{\alpha \beta }}.  \tag{113}  \label{r113}
\end{equation}%
So when $\left| \frac{k}{n}-x\right| \leq \frac{1}{n^{\beta }}$ we derived
that 
\begin{equation}
\gamma _{2k}\leq \frac{\omega _{1}\left( D_{\ast x}^{\alpha }f,\frac{1}{%
n^{\beta }}\right) _{\left[ x,b\right] }}{\alpha n^{\alpha \beta }}. 
\tag{114}  \label{r114}
\end{equation}%
Similarly we have that 
\begin{equation}
\left| \theta _{2n}\left( x\right) \right| \leq \frac{1}{\Gamma \left(
\alpha \right) }\left( \sum_{k=\left\lfloor nx\right\rfloor
+1}^{\left\lfloor nb\right\rfloor }\chi \left( nx-k\right) \gamma
_{2k}\right) =  \tag{115}  \label{r115}
\end{equation}%
\begin{equation*}
\frac{1}{\Gamma \left( \alpha \right) }\left\{ \sum_{\left\{ 
\begin{array}{c}
k=\left\lfloor nx\right\rfloor +1 \\ 
:\left| \frac{k}{n}-x\right| \leq \frac{1}{n^{\beta }}%
\end{array}%
\right. }^{\left\lfloor nb\right\rfloor }\chi \left( nx-k\right) \gamma
_{2k}+\sum_{\left\{ 
\begin{array}{c}
k=\left\lfloor nx\right\rfloor +1 \\ 
:\left| \frac{k}{n}-x\right| >\frac{1}{n^{\beta }}%
\end{array}%
\right. }^{\left\lfloor nb\right\rfloor }\chi \left( nx-k\right) \gamma
_{2k}\right\} \leq
\end{equation*}%
\begin{equation*}
\frac{1}{\Gamma \left( \alpha \right) }\left\{ \left( \sum_{\left\{ 
\begin{array}{c}
k=\left\lfloor nx\right\rfloor +1 \\ 
:\left| \frac{k}{n}-x\right| \leq \frac{1}{n^{\beta }}%
\end{array}%
\right. }^{\left\lfloor nb\right\rfloor }\chi \left( nx-k\right) \right) 
\frac{\omega _{1}\left( D_{\ast x}^{\alpha }f,\frac{1}{n^{\beta }}\right) _{%
\left[ x,b\right] }}{\alpha n^{\alpha \beta }}+\right.
\end{equation*}%
\begin{equation}
\left. \left( \sum_{\left\{ 
\begin{array}{c}
k=\left\lfloor nx\right\rfloor +1 \\ 
:\left| \frac{k}{n}-x\right| >\frac{1}{n^{\beta }}%
\end{array}%
\right. }^{\left\lfloor nb\right\rfloor }\chi \left( nx-k\right) \right)
\left\| D_{\ast x}^{\alpha }f\right\| _{\infty ,\left[ x,b\right] }\frac{%
\left( b-x\right) ^{\alpha }}{\alpha }\right\} \leq  \tag{116}  \label{r116}
\end{equation}%
\begin{equation*}
\frac{1}{\Gamma \left( \alpha +1\right) }\left\{ \frac{\omega _{1}\left(
D_{\ast x}^{\alpha }f,\frac{1}{n^{\beta }}\right) _{\left[ x,b\right] }}{%
n^{\alpha \beta }}+\right.
\end{equation*}%
\begin{equation}
\left. \left( \sum_{\left\{ 
\begin{array}{c}
k=-\infty \\ 
:\left| \frac{k}{n}-x\right| >\frac{1}{n^{\beta }}%
\end{array}%
\right. }^{\infty }\chi \left( nx-k\right) \right) \left\| D_{\ast
x}^{\alpha }f\right\| _{\infty ,\left[ x,b\right] }\left( b-x\right)
^{\alpha }\right\} \overset{\text{(\ref{r18})}}{\leq }  \tag{117}
\label{r117}
\end{equation}%
\begin{equation*}
\frac{1}{\Gamma \left( \alpha +1\right) }\left\{ \frac{\omega _{1}\left(
D_{\ast x}^{\alpha }f,\frac{1}{n^{\beta }}\right) _{\left[ x,b\right] }}{%
n^{\alpha \beta }}+\right.
\end{equation*}%
\begin{equation*}
\left. \frac{1}{2\sqrt{\pi }\left( n^{1-\beta }-2\right) e^{\left(
n^{1-\beta }-2\right) ^{2}}}\left\| D_{\ast x}^{\alpha }f\right\| _{\infty ,%
\left[ x,b\right] }\left( b-x\right) ^{\alpha }\right\} .
\end{equation*}%
So we have proved that 
\begin{equation}
\left| \theta _{2n}\left( x\right) \right| \leq \frac{1}{\Gamma \left(
\alpha +1\right) }\left\{ \frac{\omega _{1}\left( D_{\ast x}^{\alpha }f,%
\frac{1}{n^{\beta }}\right) _{\left[ x,b\right] }}{n^{\alpha \beta }}+\right.
\tag{118}  \label{r118}
\end{equation}%
\begin{equation*}
\left. \frac{1}{2\sqrt{\pi }\left( n^{1-\beta }-2\right) e^{\left(
n^{1-\beta }-2\right) ^{2}}}\left\| D_{\ast x}^{\alpha }f\right\| _{\infty ,%
\left[ x,b\right] }\left( b-x\right) ^{\alpha }\right\} .
\end{equation*}%
Therefore 
\begin{equation}
\left| \theta _{n}\left( x\right) \right| \leq \left| \theta _{1n}\left(
x\right) \right| +\left| \theta _{2n}\left( x\right) \right| \leq  \tag{119}
\label{r119}
\end{equation}%
\begin{equation}
\frac{1}{\Gamma \left( \alpha +1\right) }\left\{ \frac{\omega _{1}\left(
D_{x-}^{\alpha }f,\frac{1}{n^{\beta }}\right) _{\left[ a,x\right] }+\omega
_{1}\left( D_{\ast x}^{\alpha }f,\frac{1}{n^{\beta }}\right) _{\left[ x,b%
\right] }}{n^{\alpha \beta }}+\right.  \tag{120}  \label{r120}
\end{equation}%
\begin{equation*}
\left. \frac{1}{2\sqrt{\pi }\left( n^{1-\beta }-2\right) e^{\left(
n^{1-\beta }-2\right) ^{2}}}\left( \left\| D_{x-}^{\alpha }f\right\|
_{\infty ,\left[ a,x\right] }\left( x-a\right) ^{\alpha }+\left\| D_{\ast
x}^{\alpha }f\right\| _{\infty ,\left[ x,b\right] }\left( b-x\right)
^{\alpha }\right) \right\} .
\end{equation*}%
As in (\ref{r69}) we get that 
\begin{equation}
\left| A_{n}^{\ast }\left( \left( \cdot -x\right) ^{j}\right) \left(
x\right) \right| \leq \frac{1}{n^{\beta j}}+\left( b-a\right) ^{j}\frac{1}{2%
\sqrt{\pi }\left( n^{1-\beta }-2\right) e^{\left( n^{1-\beta }-2\right) ^{2}}%
},  \tag{121}  \label{r121}
\end{equation}%
for $j=1,...,N-1$, $\forall $ $x\in \left[ a,b\right] .$

Putting things together, we have established 
\begin{equation}
\left| A_{n}^{\ast }\left( f,x\right) -f\left( x\right) \left(
\sum_{k=\left\lceil na\right\rceil }^{\left\lfloor nb\right\rfloor }\chi
\left( nx-k\right) \right) \right| \leq  \tag{122}  \label{r122}
\end{equation}%
\begin{equation*}
\sum_{j=1}^{N-1}\frac{\left| f^{\left( j\right) }\left( x\right) \right| }{j!%
}\left[ \frac{1}{n^{\beta j}}+\left( b-a\right) ^{j}\frac{1}{2\sqrt{\pi }%
\left( n^{1-\beta }-2\right) e^{\left( n^{1-\beta }-2\right) ^{2}}}\right] +
\end{equation*}%
\begin{equation*}
\frac{1}{\Gamma \left( \alpha +1\right) }\left\{ \frac{\left( \omega
_{1}\left( D_{x-}^{\alpha }f,\frac{1}{n^{\beta }}\right) _{\left[ a,x\right]
}+\omega _{1}\left( D_{\ast x}^{\alpha }f,\frac{1}{n^{\beta }}\right) _{%
\left[ x,b\right] }\right) }{n^{\alpha \beta }}+\right.
\end{equation*}%
\begin{equation*}
\frac{1}{2\sqrt{\pi }\left( n^{1-\beta }-2\right) e^{\left( n^{1-\beta
}-2\right) ^{2}}}\cdot
\end{equation*}%
\begin{equation}
\left. \left( \left\| D_{x-}^{\alpha }f\right\| _{\infty ,\left[ a,x\right]
}\left( x-a\right) ^{\alpha }+\left\| D_{\ast x}^{\alpha }f\right\| _{\infty
,\left[ x,b\right] }\left( b-x\right) ^{\alpha }\right) \right\}
=:T_{n}\left( x\right) .  \tag{123}  \label{r123}
\end{equation}%
As a result, see (\ref{r39}), we derive 
\begin{equation}
\left| A_{n}\left( f,x\right) -f\left( x\right) \right| \leq \left(
4.019\right) T_{n}\left( x\right) ,  \tag{124}  \label{r124}
\end{equation}%
$\forall $ $x\in \left[ a,b\right] .$

We further have that 
\begin{equation}
\left\| T_{n}\right\| _{\infty }\leq \sum_{j=1}^{N-1}\frac{\left\| f^{\left(
j\right) }\right\| _{\infty }}{j!}\left[ \frac{1}{n^{\beta j}}+\left(
b-a\right) ^{j}\frac{1}{2\sqrt{\pi }\left( n^{1-\beta }-2\right) e^{\left(
n^{1-\beta }-2\right) ^{2}}}\right] +  \tag{125}  \label{r125}
\end{equation}%
\begin{equation*}
\frac{1}{\Gamma \left( \alpha +1\right) }\left\{ \frac{\left\{ \underset{%
x\in \left[ a,b\right] }{\sup }\left( \omega _{1}\left( D_{x-}^{\alpha }f,%
\frac{1}{n^{\beta }}\right) _{\left[ a,x\right] }\right) +\underset{x\in %
\left[ a,b\right] }{\sup }\left( \omega _{1}\left( D_{\ast x}^{\alpha }f,%
\frac{1}{n^{\beta }}\right) _{\left[ x,b\right] }\right) \right\} }{%
n^{\alpha \beta }}\right.
\end{equation*}%
\begin{equation*}
+\frac{1}{2\sqrt{\pi }\left( n^{1-\beta }-2\right) e^{\left( n^{1-\beta
}-2\right) ^{2}}}\left( b-a\right) ^{\alpha }\cdot
\end{equation*}%
\begin{equation*}
\left. \left\{ \left( \underset{x\in \left[ a,b\right] }{\sup }\left(
\left\| D_{x-}^{\alpha }f\right\| _{\infty ,\left[ a,x\right] }\right) +%
\underset{x\in \left[ a,b\right] }{\sup }\left( \left\| D_{\ast x}^{\alpha
}f\right\| _{\infty ,\left[ x,b\right] }\right) \right) \right\} \right\}
=:E_{n}.
\end{equation*}%
Hence it holds 
\begin{equation}
\left\| A_{n}f-f\right\| _{\infty }\leq \left( 4.019\right) E_{n}.  \tag{126}
\label{r126}
\end{equation}%
Since $f\in AC^{N}\left( \left[ a,b\right] \right) $, $N=\left\lceil \alpha
\right\rceil $, $\alpha >0$, $\alpha \notin \mathbb{N}$, $f^{\left( N\right)
}\in L_{\infty }\left( \left[ a,b\right] \right) $, $x\in \left[ a,b\right] $%
, then we get that $f\in AC^{N}\left( \left[ a,x\right] \right) $, $%
f^{\left( N\right) }\in L_{\infty }\left( \left[ a,x\right] \right) $ and $%
f\in AC^{N}\left( \left[ x,b\right] \right) $, $f^{\left( N\right) }\in
L_{\infty }\left( \left[ x,b\right] \right) .$

We have 
\begin{equation}
\left( D_{x-}^{\alpha }f\right) \left( y\right) =\frac{\left( -1\right) ^{N}%
}{\Gamma \left( N-\alpha \right) }\int_{y}^{x}\left( J-y\right) ^{N-\alpha
-1}f^{\left( N\right) }\left( J\right) dJ,  \tag{127}  \label{r127}
\end{equation}%
$\forall $ $y\in \left[ a,x\right] $ and 
\begin{equation}
\left| \left( D_{x-}^{\alpha }f\right) \left( y\right) \right| \leq \frac{1}{%
\Gamma \left( N-\alpha \right) }\left( \int_{y}^{x}\left( J-y\right)
^{N-\alpha -1}dJ\right) \left\| f^{\left( N\right) }\right\| _{\infty } 
\tag{128}  \label{r128}
\end{equation}%
\begin{equation*}
=\frac{1}{\Gamma \left( N-\alpha \right) }\frac{\left( x-y\right) ^{N-\alpha
}}{\left( N-\alpha \right) }\left\| f^{\left( N\right) }\right\| _{\infty }=
\end{equation*}%
\begin{equation*}
\frac{\left( x-y\right) ^{N-\alpha }}{\Gamma \left( N-\alpha +1\right) }%
\left\| f^{\left( N\right) }\right\| _{\infty }\leq \frac{\left( b-a\right)
^{N-\alpha }}{\Gamma \left( N-\alpha +1\right) }\left\| f^{\left( N\right)
}\right\| _{\infty }.
\end{equation*}%
That is 
\begin{equation}
\left\| D_{x-}^{\alpha }f\right\| _{\infty ,\left[ a,x\right] }\leq \frac{%
\left( b-a\right) ^{N-\alpha }}{\Gamma \left( N-\alpha +1\right) }\left\|
f^{\left( N\right) }\right\| _{\infty },  \tag{129}  \label{r129}
\end{equation}%
and 
\begin{equation}
\underset{x\in \left[ a,b\right] }{\sup }\left\| D_{x-}^{\alpha }f\right\|
_{\infty ,\left[ a,x\right] }\leq \frac{\left( b-a\right) ^{N-\alpha }}{%
\Gamma \left( N-\alpha +1\right) }\left\| f^{\left( N\right) }\right\|
_{\infty }.  \tag{130}  \label{r130}
\end{equation}%
Similarly we have 
\begin{equation}
\left( D_{\ast x}^{\alpha }f\right) \left( y\right) =\frac{1}{\Gamma \left(
N-\alpha \right) }\int_{x}^{y}\left( y-t\right) ^{N-\alpha -1}f^{\left(
N\right) }\left( t\right) dt,  \tag{131}  \label{r131}
\end{equation}%
$\forall $ $y\in \left[ x,b\right] .$

Thus we get 
\begin{equation*}
\left| \left( D_{\ast x}^{\alpha }f\right) \left( y\right) \right| \leq 
\frac{1}{\Gamma \left( N-\alpha \right) }\left( \int_{x}^{y}\left(
y-t\right) ^{N-\alpha -1}dt\right) \left\| f^{\left( N\right) }\right\|
_{\infty }\leq
\end{equation*}%
\begin{equation*}
\frac{1}{\Gamma \left( N-\alpha \right) }\frac{\left( y-x\right) ^{N-\alpha }%
}{\left( N-\alpha \right) }\left\| f^{\left( N\right) }\right\| _{\infty
}\leq \frac{\left( b-a\right) ^{N-\alpha }}{\Gamma \left( N-\alpha +1\right) 
}\left\| f^{\left( N\right) }\right\| _{\infty }.
\end{equation*}%
Hence 
\begin{equation}
\left\| D_{\ast x}^{\alpha }f\right\| _{\infty ,\left[ x,b\right] }\leq 
\frac{\left( b-a\right) ^{N-\alpha }}{\Gamma \left( N-\alpha +1\right) }%
\left\| f^{\left( N\right) }\right\| _{\infty },  \tag{132}  \label{r132}
\end{equation}%
and 
\begin{equation}
\underset{x\in \left[ a,b\right] }{\sup }\left\| D_{\ast x}^{\alpha
}f\right\| _{\infty ,\left[ x,b\right] }\leq \frac{\left( b-a\right)
^{N-\alpha }}{\Gamma \left( N-\alpha +1\right) }\left\| f^{\left( N\right)
}\right\| _{\infty }.  \tag{133}  \label{r133}
\end{equation}%
From (\ref{r82}) and (\ref{r83}) we get 
\begin{equation}
\underset{x\in \left[ a,b\right] }{\sup }\omega _{1}\left( D_{x-}^{\alpha }f,%
\frac{1}{n^{\beta }}\right) _{\left[ a,x\right] }\leq \frac{2\left\|
f^{\left( N\right) }\right\| _{\infty }}{\Gamma \left( N-\alpha +1\right) }%
\left( b-a\right) ^{N-\alpha },  \tag{134}  \label{r134}
\end{equation}%
and 
\begin{equation}
\underset{x\in \left[ a,b\right] }{\sup }\omega _{1}\left( D_{\ast
x}^{\alpha }f,\frac{1}{n^{\beta }}\right) _{\left[ x,b\right] }\leq \frac{%
2\left\| f^{\left( N\right) }\right\| _{\infty }}{\Gamma \left( N-\alpha
+1\right) }\left( b-a\right) ^{N-\alpha }.  \tag{135}  \label{r135}
\end{equation}%
So that $E_{n}<\infty .$

We finally notice that 
\begin{equation*}
A_{n}\left( f,x\right) -\dsum\limits_{j=1}^{N-1}\frac{f^{\left( j\right)
}\left( x\right) }{j!}A_{n}\left( \left( \cdot -x\right) ^{j}\right) \left(
x\right) -f\left( x\right) =\frac{A_{n}^{\ast }\left( f,x\right) }{\left(
\sum_{k=\left\lceil na\right\rceil }^{\left\lfloor nb\right\rfloor }\chi
\left( nx-k\right) \right) }
\end{equation*}%
\begin{equation*}
-\frac{1}{\left( \sum_{k=\left\lceil na\right\rceil }^{\left\lfloor
nb\right\rfloor }\chi \left( nx-k\right) \right) }\left(
\dsum\limits_{j=1}^{N-1}\frac{f^{\left( j\right) }\left( x\right) }{j!}%
A_{n}^{\ast }\left( \left( \cdot -x\right) ^{j}\right) \left( x\right)
\right) -f\left( x\right)
\end{equation*}%
\begin{equation}
=\frac{1}{\left( \sum_{k=\left\lceil na\right\rceil }^{\left\lfloor
nb\right\rfloor }\chi \left( nx-k\right) \right) }\cdot  \tag{136}
\label{r136}
\end{equation}%
\begin{equation*}
\left[ A_{n}^{\ast }\left( f,x\right) -\left( \dsum\limits_{j=1}^{N-1}\frac{%
f^{\left( j\right) }\left( x\right) }{j!}A_{n}^{\ast }\left( \left( \cdot
-x\right) ^{j}\right) \left( x\right) \right) -\left( \sum_{k=\left\lceil
na\right\rceil }^{\left\lfloor nb\right\rfloor }\chi \left( nx-k\right)
\right) f\left( x\right) \right] .
\end{equation*}%
Therefore we get 
\begin{equation*}
\left| A_{n}\left( f,x\right) -\dsum\limits_{j=1}^{N-1}\frac{f^{\left(
j\right) }\left( x\right) }{j!}A_{n}\left( \left( \cdot -x\right)
^{j}\right) \left( x\right) -f\left( x\right) \right| \leq \left(
4.019\right) \cdot
\end{equation*}%
\begin{equation}
\left| A_{n}^{\ast }\left( f,x\right) -\left( \dsum\limits_{j=1}^{N-1}\frac{%
f^{\left( j\right) }\left( x\right) }{j!}A_{n}^{\ast }\left( \left( \cdot
-x\right) ^{j}\right) \left( x\right) \right) -\left( \sum_{k=\left\lceil
na\right\rceil }^{\left\lfloor nb\right\rfloor }\chi \left( nx-k\right)
\right) f\left( x\right) \right| ,  \tag{137}  \label{r137}
\end{equation}%
$\forall $ $x\in \left[ a,b\right] .$

The proof of the theorem is now complete.
\end{proof}

Next we apply Theorem \ref{t30} for $N=1.$

\begin{corollary}
\label{c31}Let $0<\alpha ,\beta <1$, $n^{1-\beta }\geq 3,$ $f\in AC\left( %
\left[ a,b\right] \right) $, $f^{\prime }\in L_{\infty }\left( \left[ a,b%
\right] \right) $, $n\in \mathbb{N}$. Then 
\begin{equation}
\left\Vert A_{n}f-f\right\Vert _{\infty }\leq \frac{\left( 4.019\right) }{%
\Gamma \left( \alpha +1\right) }\cdot   \tag{138}  \label{r138}
\end{equation}%
\begin{equation*}
\left\{ \frac{\left( \underset{x\in \left[ a,b\right] }{\sup }\omega
_{1}\left( D_{x-}^{\alpha }f,\frac{1}{n^{\beta }}\right) _{\left[ a,x\right]
}+\underset{x\in \left[ a,b\right] }{\sup }\omega _{1}\left( D_{\ast
x}^{\alpha }f,\frac{1}{n_{\beta }}\right) _{\left[ x,b\right] }\right) }{%
n^{\alpha \beta }}\right. +
\end{equation*}%
\begin{equation*}
\frac{1}{2\sqrt{\pi }\left( n^{1-\beta }-2\right) e^{\left( n^{1-\beta
}-2\right) ^{2}}}\left( b-a\right) ^{\alpha }\cdot 
\end{equation*}%
\begin{equation*}
\left. \left( \underset{x\in \left[ a,b\right] }{\sup }\left\Vert
D_{x-}^{\alpha }f\right\Vert _{\infty ,\left[ a,x\right] }+\underset{x\in %
\left[ a,b\right] }{\sup }\left\Vert D_{\ast x}^{\alpha }f\right\Vert
_{\infty ,\left[ x,b\right] }\right) \right\} .
\end{equation*}
\end{corollary}

\begin{remark}
\label{re32}Let $0<\alpha <1$, then by (\ref{r130}), we get 
\begin{equation}
\underset{x\in \left[ a,b\right] }{\sup }\left\| D_{x-}^{\alpha }f\right\|
_{\infty ,\left[ a,x\right] }\leq \frac{\left( b-a\right) ^{1-\alpha }}{%
\Gamma \left( 2-\alpha \right) }\left\| f^{\prime }\right\| _{\infty }, 
\tag{139}  \label{r139}
\end{equation}%
and by (\ref{r133}), we obtain 
\begin{equation}
\underset{x\in \left[ a,b\right] }{\sup }\left\| D_{\ast x}^{\alpha
}f\right\| _{\infty ,\left[ x,b\right] }\leq \frac{\left( b-a\right)
^{1-\alpha }}{\Gamma \left( 2-\alpha \right) }\left\| f^{\prime }\right\|
_{\infty },  \tag{140}  \label{r140}
\end{equation}%
given that $f\in AC\left( \left[ a,b\right] \right) $ and $f^{\prime }\in
L_{\infty }\left( \left[ a,b\right] \right) .$
\end{remark}

Next we specialize to $\alpha =\frac{1}{2}.$

\begin{corollary}
\label{c33}Let $0<\beta <1$, $n^{1-\beta }\geq 3,$ $f\in AC\left( \left[ a,b%
\right] \right) $, $f^{\prime }\in L_{\infty }\left( \left[ a,b\right]
\right) $, $n\in \mathbb{N}$. Then%
\begin{equation*}
\left\Vert A_{n}f-f\right\Vert _{\infty }\leq \frac{\left( 8.038\right) }{%
\sqrt{\pi }}\cdot 
\end{equation*}%
\begin{equation*}
\left\{ \frac{\left( \underset{x\in \left[ a,b\right] }{\sup }\omega
_{1}\left( D_{x-}^{\frac{1}{2}}f,\frac{1}{n^{\beta }}\right) _{\left[ a,x%
\right] }+\underset{x\in \left[ a,b\right] }{\sup }\omega _{1}\left( D_{\ast
x}^{\frac{1}{2}}f,\frac{1}{n^{\beta }}\right) _{\left[ x,b\right] }\right) }{%
n^{\frac{\beta }{2}}}\right. +
\end{equation*}%
\begin{equation*}
\frac{1}{2\sqrt{\pi }\left( n^{1-\beta }-2\right) e^{\left( n^{1-\beta
}-2\right) ^{2}}}\sqrt{b-a}\cdot 
\end{equation*}%
\begin{equation}
\left. \left( \underset{x\in \left[ a,b\right] }{\sup }\left\Vert D_{x-}^{%
\frac{1}{2}}f\right\Vert _{\infty ,\left[ a,x\right] }+\underset{x\in \left[
a,b\right] }{\sup }\left\Vert D_{\ast x}^{\frac{1}{2}}f\right\Vert _{\infty ,%
\left[ x,b\right] }\right) \right\} ,  \tag{141}  \label{r141}
\end{equation}
\end{corollary}

\begin{remark}
\label{re34}(to Corollary \ref{c33}) Assume that 
\begin{equation}
\omega _{1}\left( D_{x-}^{\frac{1}{2}}f,\frac{1}{n^{\beta }}\right) _{\left[
a,x\right] }\leq \frac{K_{1}}{n^{\beta }},  \tag{142}  \label{r142}
\end{equation}%
and 
\begin{equation}
\omega _{1}\left( D_{\ast x}^{\frac{1}{2}}f,\frac{1}{n_{\beta }}\right) _{%
\left[ x,b\right] }\leq \frac{K_{2}}{n^{\beta }},  \tag{143}  \label{r143}
\end{equation}%
$\forall $ $x\in \left[ a,b\right] $, $\forall $ $n\in \mathbb{N}$, where $%
K_{1},K_{2}>0.$

Then for large enough $n\in \mathbb{N}$, by (\ref{r141}), we obtain 
\begin{equation}
\left\| A_{n}f-f\right\| _{\infty }\leq \frac{M}{n^{\frac{3}{2}\beta }}, 
\tag{144}  \label{r144}
\end{equation}%
for some $M>0.$

The speed of convergence in (\ref{r144}) is much higher than the
corresponding speeds achieved in (\ref{r40}), which were there $\frac{1}{%
n^{\beta }}.$
\end{remark}

\section{Complex Neural Network Approximations}

We make

\begin{remark}
\label{re35}Let $X:=\left[ a,b\right] $, $\mathbb{R}$ and $f:X\rightarrow 
\mathbb{C}$ with real and imaginary parts $f_{1},f_{2}:f=f_{1}+if_{2}$, $i=%
\sqrt{-1}$. Clearly $f$ is continuous iff $f_{1}$ and $f_{2}$ are continuous.

Also it holds 
\begin{equation}
f^{\left( j\right) }\left( x\right) =f_{1}^{\left( j\right) }\left( x\right)
+if_{2}^{\left( j\right) }\left( x\right) ,  \tag{145}  \label{r145}
\end{equation}%
for all $j=1,...,N$, given that $f_{1},f_{2}\in C^{N}\left( X\right) $, $%
N\in \mathbb{N}$.

We denote by $C_{B}\left( \mathbb{R},\mathbb{C}\right) $ the space of
continuous and bounded functions $f:\mathbb{R}\rightarrow \mathbb{C}$.
Clearly $f$ is bounded, iff both $f_{1},f_{2}$ are bounded from $\mathbb{R}$
into $\mathbb{R}$, where $f=f_{1}+if_{2}.$

Here we define 
\begin{equation}
A_{n}\left( f,x\right) :=A_{n}\left( f_{1},x\right) +iA_{n}\left(
f_{2},x\right) ,  \tag{146}  \label{r146}
\end{equation}%
and 
\begin{equation}
B_{n}\left( f,x\right) :=B_{n}\left( f_{1},x\right) +iB_{n}\left(
f_{2},x\right) .  \tag{147}  \label{r147}
\end{equation}%
We observe here that 
\begin{equation}
\left| A_{n}\left( f,x\right) -f\left( x\right) \right| \leq \left|
A_{n}\left( f_{1},x\right) -f_{1}\left( x\right) \right| +\left| A_{n}\left(
f_{2},x\right) -f_{2}\left( x\right) \right| ,  \tag{148}  \label{r148}
\end{equation}%
and 
\begin{equation}
\left\| A_{n}\left( f\right) -f\right\| _{\infty }\leq \left\| A_{n}\left(
f_{1}\right) -f_{1}\right\| _{\infty }+\left\| A_{n}\left( f_{2}\right)
-f_{2}\right\| _{\infty }.  \tag{149}  \label{r149}
\end{equation}%
Similarly we get 
\begin{equation}
\left| B_{n}\left( f,x\right) -f\left( x\right) \right| \leq \left|
B_{n}\left( f_{1},x\right) -f_{1}\left( x\right) \right| +\left| B_{n}\left(
f_{2},x\right) -f_{2}\left( x\right) \right| ,  \tag{150}  \label{r150}
\end{equation}%
and 
\begin{equation}
\left\| B_{n}\left( f\right) -f\right\| _{\infty }\leq \left\| B_{n}\left(
f_{1}\right) -f_{1}\right\| _{\infty }+\left\| B_{n}\left( f_{2}\right)
-f_{2}\right\| _{\infty }.  \tag{151}  \label{r151}
\end{equation}
\end{remark}

We present

\begin{theorem}
\label{t36}Let $f\in C\left( \left[ a,b\right] ,\mathbb{C}\right) ,$ $%
f=f_{1}+if_{2}$, $0<\alpha <1,$ $n\in \mathbb{N}$, $n^{1-\alpha }\geq 3,$ $%
x\in \left[ a,b\right] $. Then

i) 
\begin{equation}
\left| A_{n}\left( f,x\right) -f\left( x\right) \right| \leq \left(
4.019\right) \cdot  \tag{152}  \label{r152}
\end{equation}%
\begin{equation*}
\left[ \left( \omega _{1}\left( f_{1},\frac{1}{n^{\alpha }}\right) +\omega
_{1}\left( f_{2},\frac{1}{n^{\alpha }}\right) \right) +\left( \left\|
f_{1}\right\| _{\infty }+\left\| f_{2}\right\| _{\infty }\right) \frac{1}{%
\sqrt{\pi }\left( n^{1-\alpha }-2\right) e^{\left( n^{1-\alpha }-2\right)
^{2}}}\right]
\end{equation*}%
\begin{equation*}
=:\psi _{1},
\end{equation*}%
and

ii) 
\begin{equation}
\left\| A_{n}\left( f\right) -f\right\| _{\infty }\leq \psi _{1}.  \tag{153}
\label{r153}
\end{equation}
\end{theorem}

\begin{proof}
Based on Remark \ref{re35} and Theorem \ref{t12}.
\end{proof}

We give

\begin{theorem}
\label{t37}Let $f\in C_{B}\left( \mathbb{R},\mathbb{C}\right) $, $%
f=f_{1}+if_{2},$ $0<\alpha <1$, $n\in \mathbb{N}$, $n^{1-\alpha }\geq 3,$ $%
x\in \mathbb{R}$. Then

i) 
\begin{equation}
\left| B_{n}\left( f,x\right) -f\left( x\right) \right| \leq \left( \omega
_{1}\left( f_{1},\frac{1}{n^{\alpha }}\right) +\omega _{1}\left( f_{2},\frac{%
1}{n^{\alpha }}\right) \right) +  \tag{154}  \label{r154}
\end{equation}%
\begin{equation*}
\left( \left\| f_{1}\right\| _{\infty }+\left\| f_{2}\right\| _{\infty
}\right) \frac{1}{\sqrt{\pi }\left( n^{1-\alpha }-2\right) e^{\left(
n^{1-\alpha }-2\right) ^{2}}}=:\psi _{2},
\end{equation*}

ii) 
\begin{equation}
\left\| B_{n}\left( f\right) -f\right\| _{\infty }\leq \psi _{2}.  \tag{155}
\label{r155}
\end{equation}
\end{theorem}

\begin{proof}
Based on Remark \ref{re35} and Theorem \ref{t13}.
\end{proof}

Next we present a result of high order complex neural network approximation.

\begin{theorem}
\label{t38}Let $f:\left[ a,b\right] \rightarrow \mathbb{C}$, $\left[ a,b%
\right] \subset \mathbb{R}$, such that $f=f_{1}+if_{2}.$ Assume $%
f_{1},f_{2}\in C^{N}\left( \left[ a,b\right] \right) ,$ $n,N\in \mathbb{N}$, 
$n^{1-\alpha }\geq 3,$ $0<\alpha <1$, $x\in \left[ a,b\right] $. Then

i) 
\begin{equation}
\left| A_{n}\left( f,x\right) -f\left( x\right) \right| \leq \left(
4.019\right) \cdot  \tag{156}  \label{r156}
\end{equation}%
\begin{equation*}
\left\{ \sum_{j=1}^{N}\frac{\left| f_{1}^{\left( j\right) }\left( x\right)
\right| +\left| f_{2}^{\left( j\right) }\left( x\right) \right| }{j!}\left[ 
\frac{1}{n^{\alpha j}}+\frac{\left( b-a\right) ^{j}}{2\sqrt{\pi }\left(
n^{1-\alpha }-2\right) e^{\left( n^{1-\alpha }-2\right) ^{2}}}\right]
+\right.
\end{equation*}%
\begin{equation*}
\left[ \frac{\omega _{1}\left( f_{1}^{\left( N\right) },\frac{1}{n^{\alpha }}%
\right) +\omega _{1}\left( f_{2}^{\left( N\right) },\frac{1}{n^{\alpha }}%
\right) }{n^{\alpha N}N!}+\right.
\end{equation*}%
\begin{equation*}
\left. \left. \left( \frac{\left( \left\| f_{1}^{\left( N\right) }\right\|
_{\infty }+\left\| f_{2}^{\left( N\right) }\right\| _{\infty }\right) \left(
b-a\right) ^{N}}{N!\sqrt{\pi }\left( n^{1-\alpha }-2\right) e^{\left(
n^{1-\alpha }-2\right) ^{2}}}\right) \right] \right\} ,
\end{equation*}

ii) assume further $f_{1}^{\left( j\right) }\left( x_{0}\right)
=f_{2}^{\left( j\right) }\left( x_{0}\right) =0$, $j=1,...,N$, for some $%
x_{0}\in \left[ a,b\right] $, it holds 
\begin{equation}
\left| A_{n}\left( f,x_{0}\right) -f\left( x_{0}\right) \right| \leq \left(
4.019\right) \cdot  \tag{157}  \label{r157}
\end{equation}%
\begin{equation*}
\left[ \frac{\omega _{1}\left( f_{1}^{\left( N\right) },\frac{1}{n^{\alpha }}%
\right) +\omega _{1}\left( f_{2}^{\left( N\right) },\frac{1}{n^{\alpha }}%
\right) }{n^{\alpha N}N!}+\right.
\end{equation*}%
\begin{equation*}
\left. \left( \frac{\left( \left\| f_{1}^{\left( N\right) }\right\| _{\infty
}+\left\| f_{2}^{\left( N\right) }\right\| _{\infty }\right) \left(
b-a\right) ^{N}}{N!\sqrt{\pi }\left( n^{1-\alpha }-2\right) e^{\left(
n^{1-\alpha }-2\right) ^{2}}}\right) \right] ,
\end{equation*}%
notice here the extremely high rate of convergence at $n^{-\left( N+1\right)
\alpha },$

iii) 
\begin{equation}
\left\| A_{n}\left( f\right) -f\right\| _{\infty }\leq \left( 4.019\right)
\cdot  \tag{158}  \label{r158}
\end{equation}%
\begin{equation*}
\left\{ \sum_{j=1}^{N}\frac{\left( \left\| f_{1}^{\left( j\right) }\right\|
_{\infty }+\left\| f_{2}^{\left( j\right) }\right\| _{\infty }\right) }{j!}%
\left[ \frac{1}{n^{\alpha j}}+\frac{\left( b-a\right) ^{j}}{2\sqrt{\pi }%
\left( n^{1-\alpha }-2\right) e^{\left( n^{1-\alpha }-2\right) ^{2}}}\right]
+\right.
\end{equation*}%
\begin{equation*}
\left[ \frac{\left( \omega _{1}\left( f_{1}^{\left( N\right) },\frac{1}{%
n^{\alpha }}\right) +\omega _{1}\left( f_{2}^{\left( N\right) },\frac{1}{%
n^{\alpha }}\right) \right) }{n^{\alpha N}N!}+\right.
\end{equation*}%
\begin{equation*}
\left. \left. \frac{\left( \left\| f_{1}^{\left( N\right) }\right\| _{\infty
}+\left\| f_{2}^{\left( N\right) }\right\| _{\infty }\right) \left(
b-a\right) ^{N}}{N!\sqrt{\pi }\left( n^{1-\alpha }-2\right) e^{\left(
n^{1-\alpha }-2\right) ^{2}}}\right] \right\} .
\end{equation*}
\end{theorem}

\begin{proof}
Based on Remark \ref{re35} and Theorem \ref{t16}.
\end{proof}

We continue with high order complex fractional neural network approximation.

\begin{theorem}
\label{t39}Let $f:\left[ a,b\right] \rightarrow \mathbb{C}$, $\left[ a,b%
\right] \subset \mathbb{R}$, such that $f=f_{1}+if_{2};$ $\alpha >0$, $%
N=\left\lceil \alpha \right\rceil $, $\alpha \notin \mathbb{N}$, $0<\beta <1$%
, $x\in \left[ a,b\right] ,$ $n\in \mathbb{N}$, $n^{1-\beta }\geq 3.$ Assume 
$f_{1},f_{2}\in AC^{N}\left( \left[ a,b\right] \right) ,$ with $%
f_{1}^{\left( N\right) },f_{2}^{\left( N\right) }\in L_{\infty }\left( \left[
a,b\right] \right) .$ Then

i) assume further $f_{1}^{\left( j\right) }\left( x\right) =f_{2}^{\left(
j\right) }\left( x\right) =0$, $j=1,...,N-1,$ we have 
\begin{equation*}
\left\vert A_{n}\left( f,x\right) -f\left( x\right) \right\vert \leq \frac{%
\left( 4.019\right) }{\Gamma \left( \alpha +1\right) }\cdot 
\end{equation*}%
\begin{equation*}
\left\{ \frac{1}{n^{\alpha \beta }}\left[ \left( \omega _{1}\left(
D_{x-}^{\alpha }f_{1},\frac{1}{n^{\beta }}\right) _{\left[ a,x\right]
}+\omega _{1}\left( D_{\ast x}^{\alpha }f_{1},\frac{1}{n^{\beta }}\right) _{%
\left[ x,b\right] }\right) +\right. \right. 
\end{equation*}%
\begin{equation*}
\left. \left( \omega _{1}\left( D_{x-}^{\alpha }f_{2},\frac{1}{n^{\beta }}%
\right) _{\left[ a,x\right] }+\omega _{1}\left( D_{\ast x}^{\alpha }f_{2},%
\frac{1}{n^{\beta }}\right) _{\left[ x,b\right] }\right) \right] +
\end{equation*}%
\begin{equation*}
\frac{1}{2\sqrt{\pi }\left( n^{1-\beta }-2\right) e^{\left( n^{1-\beta
}-2\right) ^{2}}}\cdot 
\end{equation*}%
\begin{equation*}
\left[ \left( \left\Vert D_{x-}^{\alpha }f_{1}\right\Vert _{\infty ,\left[
a,x\right] }\left( x-a\right) ^{\alpha }+\left\Vert D_{\ast x}^{\alpha
}f_{1}\right\Vert _{\infty ,\left[ x,b\right] }\left( b-x\right) ^{\alpha
}\right) +\right. 
\end{equation*}%
\begin{equation}
\left. \left. \left( \left\Vert D_{x-}^{\alpha }f_{2}\right\Vert _{\infty ,
\left[ a,x\right] }\left( x-a\right) ^{\alpha }+\left\Vert D_{\ast
x}^{\alpha }f_{2}\right\Vert _{\infty ,\left[ x,b\right] }\left( b-x\right)
^{\alpha }\right) \right] \right\} ,  \tag{159}  \label{r159}
\end{equation}%
when $\alpha >1$ notice here the extremely high rate of convergence at $%
n^{-\left( \alpha +1\right) \beta },$

ii) 
\begin{equation*}
\left\vert A_{n}\left( f,x\right) -f\left( x\right) \right\vert \leq \left(
4.019\right) \cdot \left\{ \sum_{j=1}^{N-1}\frac{\left( \left\vert
f_{1}^{\left( j\right) }\left( x\right) \right\vert +\left\vert
f_{2}^{\left( j\right) }\left( x\right) \right\vert \right) }{j!}\right. 
\end{equation*}%
\begin{equation*}
\left\{ \frac{1}{n^{\beta j}}+\frac{\left( b-a\right) ^{j}}{2\sqrt{\pi }%
\left( n^{1-\beta }-2\right) e^{\left( n^{1-\beta }-2\right) ^{2}}}\right\} +
\end{equation*}%
\begin{equation*}
\frac{1}{\Gamma \left( \alpha +1\right) }\left\{ \frac{1}{n^{\alpha \beta }}%
\left[ \left( \omega _{1}\left( D_{x-}^{\alpha }f_{1},\frac{1}{n^{\beta }}%
\right) _{\left[ a,x\right] }+\omega _{1}\left( D_{\ast x}^{\alpha }f_{1},%
\frac{1}{n^{\beta }}\right) _{\left[ x,b\right] }\right) +\right. \right. 
\end{equation*}%
\begin{equation*}
\left. \left( \omega _{1}\left( D_{x-}^{\alpha }f_{2},\frac{1}{n^{\beta }}%
\right) _{\left[ a,x\right] }+\omega _{1}\left( D_{\ast x}^{\alpha }f_{2},%
\frac{1}{n^{\beta }}\right) _{\left[ x,b\right] }\right) \right] +
\end{equation*}%
\begin{equation*}
\frac{1}{2\sqrt{\pi }\left( n^{1-\beta }-2\right) e^{\left( n^{1-\beta
}-2\right) ^{2}}}\cdot 
\end{equation*}%
\begin{equation*}
\left[ \left( \left\Vert D_{x-}^{\alpha }f_{1}\right\Vert _{\infty ,\left[
a,x\right] }\left( x-a\right) ^{\alpha }+\left\Vert D_{\ast x}^{\alpha
}f_{1}\right\Vert _{\infty ,\left[ x,b\right] }\left( b-x\right) ^{\alpha
}\right) +\right. 
\end{equation*}%
\begin{equation}
\left. \left. \left( \left\Vert D_{x-}^{\alpha }f_{2}\right\Vert _{\infty ,
\left[ a,x\right] }\left( x-a\right) ^{\alpha }+\left\Vert D_{\ast
x}^{\alpha }f_{2}\right\Vert _{\infty ,\left[ x,b\right] }\left( b-x\right)
^{\alpha }\right) \right] \right\} ,  \tag{160}  \label{r160}
\end{equation}%
and

iii) 
\begin{equation}
\left\| A_{n}\left( f\right) -f\right\| _{\infty }\leq \left( 4.019\right)
\cdot  \tag{158}
\end{equation}%
\begin{equation*}
\left\{ \sum_{j=1}^{N-1}\frac{\left( \left\| f_{1}^{\left( j\right)
}\right\| _{\infty }+\left\| f_{2}^{\left( j\right) }\right\| _{\infty
}\right) }{j!}\left\{ \frac{1}{n^{\beta j}}+\frac{\left( b-a\right) ^{j}}{2%
\sqrt{\pi }\left( n^{1-\beta }-2\right) e^{\left( n^{1-\beta }-2\right) ^{2}}%
}\right\} +\right.
\end{equation*}%
\begin{equation*}
\frac{1}{\Gamma \left( \alpha +1\right) }\left\{ \frac{1}{n^{\alpha \beta }}%
\left\{ \left[ \underset{x\in \left[ a,b\right] }{\sup }\omega _{1}\left(
D_{x-}^{\alpha }f_{1},\frac{1}{n^{\beta }}\right) _{\left[ a,x\right] }+%
\underset{x\in \left[ a,b\right] }{\sup }\omega _{1}\left( D_{\ast
x}^{\alpha }f_{1},\frac{1}{n^{\beta }}\right) _{\left[ x,b\right] }+\right.
\right. \right.
\end{equation*}%
\begin{equation*}
\left. \left. \underset{x\in \left[ a,b\right] }{\sup }\omega _{1}\left(
D_{x-}^{\alpha }f_{2},\frac{1}{n^{\beta }}\right) _{\left[ a,x\right] }+%
\underset{x\in \left[ a,b\right] }{\sup }\omega _{1}\left( D_{\ast
x}^{\alpha }f_{2},\frac{1}{n^{\beta }}\right) _{\left[ x,b\right] }\right]
\right\} +
\end{equation*}%
\begin{equation*}
\frac{\left( b-a\right) ^{\alpha }}{2\sqrt{\pi }\left( n^{1-\beta }-2\right)
e^{\left( n^{1-\beta }-2\right) ^{2}}}\cdot
\end{equation*}%
\begin{equation*}
\left[ \left( \underset{x\in \left[ a,b\right] }{\sup }\left\|
D_{x-}^{\alpha }f_{1}\right\| _{\infty ,\left[ a,x\right] }+\underset{x\in %
\left[ a,b\right] }{\sup }\left\| D_{\ast x}^{\alpha }f_{1}\right\| _{\infty
,\left[ x,b\right] }\right) +\right.
\end{equation*}%
\begin{equation}
\left. \left. \left( \underset{x\in \left[ a,b\right] }{\sup }\left\|
D_{x-}^{\alpha }f_{2}\right\| _{\infty ,\left[ a,x\right] }+\underset{x\in %
\left[ a,b\right] }{\sup }\left\| D_{\ast x}^{\alpha }f_{2}\right\| _{\infty
,\left[ x,b\right] }\right) \right] \right\} .  \tag{161}  \label{r161}
\end{equation}%
Above, when $N=1$ the sum $\sum_{j=1}^{N-1}\cdot =0$.

As we see here we obtain fractionally type pointwise and uniform convergence
with rates of complex $A_{n}\rightarrow I$ the unit operator, as $%
n\rightarrow \infty .$
\end{theorem}

\begin{proof}
Using Theorem \ref{t30} and Remark \ref{re35}.
\end{proof}

We need

\begin{definition}
\label{d40}Let $f\in C_{B}\left( \mathbb{R},\mathbb{C}\right) ,$ with $%
f=f_{1}+if_{2}$. We define 
\begin{eqnarray}
C_{n}\left( f,x\right) &:&=C_{n}\left( f_{1},x\right) +iC_{n}\left(
f_{2},x\right) ,  \notag \\
D_{n}\left( f,x\right) &:&=D_{n}\left( f_{1},x\right) +iD_{n}\left(
f_{2},x\right) ,\text{ \ \ }\forall \text{ }x\in \mathbb{R}\text{, }n\in 
\mathbb{N}.  \TCItag{162}  \label{r162}
\end{eqnarray}
\end{definition}

We finish with

\begin{theorem}
\label{t41}Let $f\in C_{B}\left( \mathbb{R},\mathbb{C}\right) ,$ $%
f=f_{1}+if_{2}$, $0<\alpha <1$, $n\in \mathbb{N}$, $n^{1-\alpha }\geq 3$, $%
x\in \mathbb{R}$. Then

i) 
\begin{equation*}
\left\{ 
\begin{array}{c}
\left| C_{n}\left( f,x\right) -f\left( x\right) \right| \\ 
\left| D_{n}\left( f,x\right) -f\left( x\right) \right|%
\end{array}%
\right. \leq \left( \omega _{1}\left( f_{1},\frac{1}{n}+\frac{1}{n^{\alpha }}%
\right) +\omega _{1}\left( f_{2},\frac{1}{n}+\frac{1}{n^{\alpha }}\right)
\right)
\end{equation*}%
\begin{equation}
+\frac{\left( \left\| f_{1}\right\| _{\infty }+\left\| f_{2}\right\|
_{\infty }\right) }{\sqrt{\pi }\left( n^{1-\alpha }-2\right) e^{\left(
n^{1-\alpha }-2\right) ^{2}}}=:\mu _{3n}\left( f_{1},f_{2}\right) , 
\tag{163}  \label{r163}
\end{equation}%
and

ii) 
\begin{equation}
\left\{ 
\begin{array}{c}
\left\| C_{n}\left( f\right) -f\right\| _{\infty } \\ 
\left\| D_{n}\left( f\right) -f\right\| _{\infty }%
\end{array}%
\right. \leq \mu _{3n}\left( f_{1},f_{2}\right) .  \tag{164}  \label{r164}
\end{equation}
\end{theorem}

\begin{proof}
By Theorems \ref{tt14}, \ref{tt15}, also see (162).
\end{proof}


\begin{thebibliography}{99}
\bibitem{1} M. Abramowitz, I.A. Stegun, eds., \textit{Handbook of
Mathematical Functions with Formulas, Graphs, and Mathematical Tables}, New
York, Dover Publications, 1972.

\bibitem{2} G.A. Anastassiou, \textit{Rate of convergence of some neural
network operators to the unit-univariate case}, J. Math. Anal. Appli. 212
(1997), 237-262.

\bibitem{3} G.A. Anastassiou, \textit{Quantitative Approximations},
Chapman\&Hall/CRC, Boca Raton, New York, 2001.

\bibitem{4} G.A. Anastassiou, \textit{On Right Fractional Calculus}, Chaos,
solitons and fractals, 42 (2009), 365-376.

\bibitem{5} G.A. Anastassiou, \textit{Fractional Differentiation Inequalities%
}, Springer, New York, 2009.

\bibitem{6} G.A. Anastassiou, \textit{Fractional Korovkin theory}, Chaos,
Solitons \& Fractals, Vol. 42, No. 4 (2009), 2080-2094.

\bibitem{7} G.A. Anastassiou, \textit{Inteligent Systems: Approximation by
Artificial Neural Networks}, Intelligent Systems Reference Library, Vol. 19,
Springer, Heidelberg, 2011.

\bibitem{8} G.A. Anastassiou, \textit{Fractional representation formulae and
right fractional inequalities}, Mathematical and Computer Modelling, Vol.
54, no. 11-12 (2011), 3098-3115.

\bibitem{9} G.A. Anastassiou, \textit{Univariate hyperbolic tangent neural
network approximation}, Mathematics and Computer Modelling, 53(2011),
1111-1132.

\bibitem{10} G.A. Anastassiou, \textit{Multivariate hyperbolic tangent
neural network approximation}, Computers and Mathematics 61(2011), 809-821.

\bibitem{11} G.A. Anastassiou, \textit{Multivariate sigmoidal neural network
approximation}, Neural Networks 24(2011), 378-386.

\bibitem{12} G.A. Anastassiou, \textit{Univariate sigmoidal neural network
approximation}, J. of Computational Analysis and Applications, Vol. 14, No.
4, 2012, 659-690.

\bibitem{13} G.A. Anastassiou, \textit{Fractional neural network
approximation}, Computers and Mathematics with Applications, 64 (2012),
1655-1676.

\bibitem{14} L.C. Andrews, \textit{Special Functions of Mathematics for
Engineers}, Second edition, Mc Graw-Hill, New York, 1992.

\bibitem{15} Z. Chen and F. Cao, \textit{The approximation operators with
sigmoidal functions}, Computers and Mathematics with Applications, 58
(2009), 758-765.

\bibitem{16} K. Diethelm, \textit{The Analysis of Fractional Differential
Equations}, Lecture Notes in Mathematics 2004, Springer-Verlag, Berlin,
Heidelberg, 2010.

\bibitem{17} A.M.A. El-Sayed and M. Gaber, \textit{On the finite Caputo and
finite Riesz derivatives}, Electronic Journal of Theoretical Physics, Vol.
3, No. 12 (2006), 81-95.

\bibitem{18} G.S. Frederico and D.F.M. Torres, \textit{Fractional Optimal
Control in the sense of Caputo and the fractional Noether's theorem},
International Mathematical Forum, Vol. 3, No. 10 (2008), 479-493.

\bibitem{19} S. Haykin, \textit{Neural Networks: A Comprehensive Foundation }%
(2 ed.), Prentice Hall, New York, 1998.

\bibitem{20} W. McCulloch and W. Pitts, \textit{A logical calculus of the
ideas immanent in nervous activity}, Bulletin of Mathematical Biophysics, 7
(1943), 115-133.

\bibitem{21} T.M. Mitchell, \textit{Machine Learning}, WCB-McGraw-Hill, New
York, 1997.

\bibitem{22} D.S. Mitrinovic, \textit{Analytical Inequalities},
Springer-Verlag, New York, Heidelberg, 1970.

\bibitem{23} S.G. Samko, A.A. Kilbas and O.I. Marichev, \textit{Fractional
Integrals and Derivatives, Theory and Applications}, (Gordon and Breach,
Amsterdam, 1993) [English translation from the Russian, Integrals and
Derivatives of Fractional Order and Some of Their Applications (Nauka i
Tekhnika, Minsk, 1987)].
\end{thebibliography}
\end{document}